\documentclass[11pt]{amsart}
\usepackage{amsmath, amsthm, amssymb, amsfonts}
\usepackage[normalem]{ulem}
\usepackage{hyperref}
\usepackage{verbatim} 
\usepackage{enumitem} 

\usepackage{mathtools}

\DeclarePairedDelimiter\floor{\lfloor}{\rfloor}

\usepackage{tikz-cd}
\usetikzlibrary{arrows}
\usepackage{lmodern}
\usetikzlibrary{decorations.pathmorphing}
\tikzset{snake it/.style={decorate, decoration=snake}}

\theoremstyle{plain}
\newtheorem{thm}{Theorem}
\newtheorem{cor}{Corollary}
\newtheorem{lemma}{Lemma}

\newtheorem{prop}{Proposition}
\newtheorem{conjecture}{Conjecture}  

\theoremstyle{definition}
\newtheorem{defn}{Definition}
\newtheorem{example}{Example}

\theoremstyle{remark}
\newtheorem{rmk}{Remark}

\newcommand{\BC}{{\mathbb{C}}}

\newcommand{\BQ}{{\mathbb{Q}}}
\newcommand{\BR}{{\mathbb{R}}}

\newcommand{\BZ}{{\mathbb{Z}}}

\newcommand{\CB}{{\mathcal B}}
\newcommand{\CC}{{\mathcal C}}
\newcommand{\CD}{{\mathcal D}}

\newcommand{\CK}{{\mathcal K}}
\newcommand{\CL}{{\mathcal L}}
\newcommand{\CM}{{\mathcal M}}

\newcommand{\CO}{{\mathcal O}}

\newcommand{\CU}{{\mathcal U}}

\newcommand{\CX}{{\mathcal X}}

\newcommand{\Fq}{{\mathfrak{q}}}

\newcommand{\blangle}{\big\langle}
\newcommand{\brangle}{\big\rangle}
\newcommand{\Blangle}{\Big\langle}
\newcommand{\Brangle}{\Big\rangle}

\newcommand{\pt}{{\mathsf{p}}}
\newcommand{\ch}{{\mathrm{ch}}}
\newcommand{\td}{{\mathrm{td}}}
\newcommand{\bw}{{\mathbf{w}}}

\renewcommand{\div}{\mathsf{div}}

\DeclareFontFamily{OT1}{rsfs}{}
\DeclareFontShape{OT1}{rsfs}{n}{it}{<-> rsfs10}{}
\DeclareMathAlphabet{\curly}{OT1}{rsfs}{n}{it}

\newcommand\Hom{\operatorname{Hom}}

\newcommand{\p}{\mathbb{P}}

\newcommand{\Mbar}{{\overline M}}

\newcommand{\Pic}{\mathop{\rm Pic}\nolimits}

\newcommand\ev{\operatorname{ev}}

\newcommand{\Hilb}{\mathsf{Hilb}}

\newcommand{\Mod}{\mathsf{Mod}}

\newcommand{\id}{\mathrm{id}}
\newcommand{\wt}{{\mathsf{wt}}}

\newcommand{\SO}{\mathrm{SO}}

\newcommand{\GL}{\mathrm{GL}}

\newcommand{\Mon}{\mathsf{Mon}}

\newcommand{\rank}{\mathrm{rank}}
\newcommand{\NL}{\mathsf{NL}}
\newcommand{\Gr}{\mathsf{Gr}}

\def\cD{{\mathcal{D}}}

\def\cN{{\mathcal{N}}}
\def\cO{{\mathcal{O}}}

\def\cQ{{\mathcal{Q}}}

\def\cU{{\mathcal{U}}}

\def\bP{{\mathbf{P}}}

\def\bw#1{\mathchoice%
 {\textstyle{\bigwedge\mkern-4.5mu^{#1}\mkern1mu}}%
 {\textstyle{\bigwedge\mkern-4.5mu^{#1}\mkern1mu}}%
 {\scriptstyle{\bigwedge\mkern-5mu^{#1}}}%
 {\scriptscriptstyle{\bigwedge\mkern-5mu^{#1}}}%
}
\def\longarrow#1#2{\mathchoice{#2}{#1}{#1}{#1}}
\def\to{\longarrow{\rightarrow}{\longrightarrow}}

\def\into{\longarrow{\hookrightarrow}{\lhook\joinrel\longrightarrow}}

\let\shortmapsto\mapsto
\def\mapsto{\longarrow{\shortmapsto}{\longmapsto}}

\def\setmid#1#2{{\left\{{#1}\;\middle|\;{#2}\right\}}}

\def\Bl{\operatorname{Bl}}
\def\Gr{\operatorname{Gr}}
\def\Hom{\operatorname{Hom}}
\def\KKK{\mathrm{K3}}

\begin{document}
\baselineskip=14pt

\title[Gromov-Witten theory and Noether-Lefschetz theory]{
Gromov-Witten theory and\\Noether-Lefschetz theory\\for
holomorphic-symplectic varieties}

\author{Georg Oberdieck,\\
with an appendix by Jieao Song}
\address{University of Bonn}
\email{georgo@math.uni-bonn.de}

\begin{abstract}
We use Noether-Lefschetz theory to study the reduced Gro\-mov--Witten invariants of a holomorphic-symplectic variety of $K3^{[n]}$-type. This yields strong evidence for a new conjectural formula that expresses Gromov-Witten invariants of this geometry for arbitrary classes in terms of primitive classes. The formula generalizes an earlier conjecture by Pandharipande and the author for K3 surfaces. Using Gromov-Witten techniques we also determine the generating series of Noether-Lefschetz numbers of a general pencil of Debarre-Voisin varieties. This reproves and extends a result of Debarre, Han, O'Grady and Voisin on HLS divisors on the moduli space of Debarre-Voisin fourfolds. 
\end{abstract}
%
%
%
\maketitle
\setcounter{tocdepth}{1} 
\tableofcontents
\setcounter{section}{-1}
\section{Introduction}
\subsection{K3 surfaces} \label{sec:intro K3}
Gromov-Witten theory is the intersection theory of the moduli space $\Mbar_{g,n}(X,\beta)$ of stable maps to a target $X$ in degree $\beta \in H_2(X,\BZ)$.
If $X$ carries a holomorphic symplectic form, the virtual fundamental class of the moduli space vanishes.
Instead Gromov-Witten theory is defined through the 
reduced virtual fundamental class 
\[ [ \Mbar_{g,n}(X,\beta) ]^{\text{red}} \in A_{\ast}(\Mbar_{g,n}(X,\beta)). \]
When working with reduced Gromov-Witten invariants, one observes in many examples the following dichotomy:
\begin{enumerate}
\item[1.] The invariants are notoriously difficult to compute, in particular if the class $\beta$ is not primitive.
\item[2.] The structure of the invariants is simpler than for general target~$X$.
That is, the invariants have additional non-geometric symmetries
such as the independence (understood correctly) from the divisibility of the curve class.
\end{enumerate}
Physicists would say that $X$ has additional super-symmetry, which should explain this phenomenon.
A mathematical explanation is unfortunately very much missing so far.

As an example, let us consider a K3 surface $S$ and the Hodge integral
\[ R_{g,\beta} = \int_{ [ \Mbar_{g,n}(S,\beta) ]^{\text{red}} } (-1)^g \lambda_g. \]
We can formally subtract "multiple cover contributions" from $\beta$ by\footnote{The notation
$\widetilde{r}_{g,\beta}$ in choosen to clearly distinguish from the Gopakumar-Vafa BPS invariants $r_{g,\beta}$
which appear in \cite{PT_KKV}. The relationship between these two sets of invariants is discussed in Appendix~\ref{sec:comparision with GV}.}
\[ \widetilde{r}_{g,\beta} = \sum_{k | \beta} k^{2g-3} \mu(k) R_{g, \beta/k} \]
where we have used the M\"obius function
\[ \mu(k) =
\begin{cases}
(-1)^{\ell} & \text{ if } k = p_1 \cdots p_{\ell} \text{ for distinct primes } p_i \\
0 & \text{ else }.
\end{cases}
\]
By deformation invariance, $\widetilde{r}_{g,\beta}$ depends on $(S,\beta)$
only through the divisibility $m = \div(\beta)$ and the square $s = \beta \cdot \beta$.
One writes
\[ \widetilde{r}_{g,\beta} = \widetilde{r}_{g,m,s}. \]

The following remarkable result by Pandharipande and Thomas shows that the invariants $\widetilde{r}_{g,\beta}$ do not depend on the divisibility.
\begin{thm}[\cite{PT_KKV}] \label{thm:PT} For all $g,m,s$ we have: $\widetilde{r}_{g,m,s} = \widetilde{r}_{g,1,s}$. \end{thm}

The calculation of the primitive invariants $\widetilde{r}_{g,1,s}$ is much easier compared to the imprimitive case and was performed first in \cite{MPT}.
Several other proofs are available in the literature by now. Together with the theorem this yields a complete determination of $\widetilde{r}_{g, \beta}$,
called the Katz-Klemm-Vafa formula \cite{KKV}.

In \cite{K3xE} it was conjectured that more generally \emph{any} Gromov-Witten invariant of a K3 surface is independent
of the divisibility, after subtracting multiple covers.
The divisibility $2$ case was solved recently \cite{BB}, but the general case remains a challenge.

\subsection{Holomorphic-symplectic varieties}
A smooth projective variety $X$ is (irreducible) holomorphic-symplectic if it is simply connected and 
the space of holomorphic $2$-forms $H^0(X,\Omega_X^2)$ is spanned by a symplectic form.
These varieties can be viewed as higher-dimensional analogues of K3 surfaces.
For example, the cohomology $H^2(X,\BZ)$ carries a canonical non-de\-generate integer-valued quadratic form.
The prime example of a holomorphic-symplectic variety is the Hilbert scheme of $n$ points of a K3 surface and its deformations,
which we call varieties of $K3^{[n]}$ type.\footnote{We also allow the case $n=1$ below, that is our formulas apply also to the case of K3 surfaces,
where they reduce to \cite{K3xE}.}

We conjecture in this paper
that the Gromov-Witten theory of $K3^{[n]}$-type varieties
is independent of the divisibility of the curve class,
made precise in the following sense:
Let $\beta \in H_2(X,\BZ)$ be an effective (hence non-zero) curve class, and consider the Gromov-Witten class
\[ \CC_{g,N,\beta}(\alpha) = \ev_{\ast}\left( \tau^{\ast}(\alpha) \cap [ \Mbar_{g,N}(X,\beta) ]^{\text{red}} \right) \in H^{\ast}(X^N) \]
where $\tau : \Mbar_{g,N}(X,\beta) \to \Mbar_{g,N}$ is the forgetful morphism to the moduli space of curves
and $\alpha \in H^{\ast}(\Mbar_{g,N})$ is a tautological class \cite{FP}.
The classes $\CC_{g,N,\beta}(\alpha)$ encode the full numerical Gromov-Witten theory of $X$.

We formally subtract the multiple cover contributions from this class:
\begin{equation} \label{abc}
\mathsf{c}_{g,N,\beta}(\alpha) = \sum_{k|\beta} \mu(k) k^{3g-3+N-\deg(\alpha)} (-1)^{[\beta] + [\beta/k]} \CC_{g,N,\beta/k}(\alpha)
\end{equation}
where we use that the residue of $\beta$ with respect to the quadratic form 
\[ [\beta] \in H_2(X,\BZ)/ H^2(X,\BZ) \]
can be canonically identified up to multiplication by $\pm 1$ with an element of $\BZ / (2n-2) \BZ$, see Section~\ref{subsec:curve classes}.

Let $X'$ be any variety of $K3^{[n]}$-type and let
\[ \varphi : H^2(X,\BR) \to H^2(X',\BR) \]
be any \emph{real isometry} such that $\varphi(\beta) \in H_2(X',\BZ)$ is a \emph{primitive} effective curve class satisfying
\[ \pm [ \varphi(\beta) ] = \pm [\beta] \ \text{ in } \ \BZ/(2n-2) \BZ. \]
Extend $\varphi$ to the full cohomology as a parallel transport lift (Section~\ref{Subsec:Parallel Transport lifts})
\[ \varphi : H^{\ast}(X,\BR) \to H^{\ast}(X', \BR). \]

The following is the main conjecture:
\begin{conjecture} \label{conjecture:MC_intro}
\[ \mathsf{c}_{g,N,\beta}(\alpha) = \varphi^{-1}\left( \CC_{g,N,\varphi(\beta)}(\alpha) \right) \]
\end{conjecture}

The right hand side of the conjecture is given by the Gromov-Witten theory for a primitive class.
Hence the conjecture reduces calculations in imprimitive classes (which are hard) to those for primitive ones (which are easier).
A different but equivalent version of the conjecture is formulated in Section~\ref{subsec:mc conjecture}.
The equivalent version shows that the above reduces for K3 surfaces to the conjecture \cite[Conj.C2]{K3xE}.

The moduli space $\Mbar_{g,0}(X,\beta)$ is of reduced virtual dimension
\begin{equation} (\dim (X) - 3)(1-g) + 1.\label{vd332t} \end{equation}
Hence the Gromov-Witten theory of $X$ of $K3^{[n]}$-type vanishes for $g>1$ if $n \geq 3$,
and for $g > 2$ if $n=2$.
Moreover, for $g=1$ the virtual dimension \eqref{vd332t} is always one-dimensional, so relatively small.
Hence in dimension $>2$ Conjecture~\ref{conjecture:MC_intro} mainly concerns the genus zero theory.
Our main evidence in genus $g>0$ comes from the case of K3 surfaces \cite{PT_KKV, BB, K3xE}
and the remarkable fact that there is a single formula which governs all $K3^{[n]}$ at the same time.

In genus $0$, up to the sign $(-1)^{[\beta]+[\beta/k]}$, \eqref{abc} is precisely the formula that defines the BPS numbers of a Calabi-Yau manifold in
terms of Gromov-Witten invariants.
However, the appearence of the sign is a new feature, particular to the holomorphic-symplectic case.
For example, it does not appear in the definition of genus $0$ BPS numbers of Calabi-Yau 4-folds as given by Klemm-Pandharipande \cite{KP}.
We expect a similar multiple cover formula to hold for all holomorphic-symplectic varieties.
What stops us from formulating it is that the precise term that generalizes the sign
is not clear (aside from that there would be no evidence available at all).

In the appendix we also formulate a multiple cover rule for abelian surfaces,
extending a proposal for abelian varieties in \cite{BOPY}.

\subsection{Noether-Lefschetz theory}
There are three types of invariants
associated to a $1$-parameter family $\pi : \CX \to C$
of quasi-polarized holomorphic-symplectic varieties: 
\begin{enumerate}
\item[(i)] the Noether-Lefschetz numbers of $\pi$,
\item[(ii)] the Gromov-Witten invariants of $\CX$ in fiber classes,
\item[(iii)] the reduced Gromov-Witten invariants of a holomorphic-symplectic fiber of the family.
\end{enumerate}
We refer to Section~\ref{sec:NL theory} for the definition of a $1$-parameter family of quasi-polarized holo\-morphic-symp\-lectic varieties and its Noether-Lefschetz numbers.
By a result of Maulik and Pandharipande \cite{MP}, there is a geometric relation
intertwining these three invariants.
This relation (for a carefully selected family $\pi$) was used in \cite{KMPS}
to prove Theorem~\ref{thm:PT} in genus $0$, and than later in \cite{PT_KKV} in the general case.
Roughly, for a nice family the relation becomes invertible, and reduces the problem to considering
ordinary Gromov-Witten invariants
of a K3-fibered threefold which then can be attacked by more standard methods.

In this paper we follow the same strategy for holomorphic symplectic varieties of $K3^{[n]}$-type.
We first discuss the Maulik-Pandharipande relation in this case (Section~\ref{sec:GWNL theory}),
and then apply it in two cases:
\begin{itemize}
\item[(i)] A generic pencil of Fano varieties of a cubic fourfold \cite{BD} :
\[ \CX \subset \Gr(2,6) \times \p^1, \quad \pi : \CX \to \p^1. \]
\item[(ii)] A generic pencil of Debarre-Voisin varieties  \cite{DV}:
\[ \CX \subset \Gr(6,10) \times \p^1, \quad \pi : \CX \to \p^1. \]
\end{itemize}

The examples are choosen such that the Gromov-Witten invariants of $\CX$ can be computed using
mirror symmetry \cite{CFK}. For the family of Fano varieties, the Noether-Lefschetz numbers have already been computed by Li and Zhang \cite{LZ}.
The Gromov-Witten/Noether-Lefschetz relation for the Fano family then yields an
explicit infinite family of relations that need to hold for Conjecture~\ref{conjecture:MC_intro} to be true.
Using a computer we checked that these relations are satisfied up to degree $d \leq 38$ (with respect to the Pl\"ucker polarization).
The relations involve curve classes of both high self-intersection and high divisibility.
Together with previously known cases we also obtain the following:
\begin{prop} \label{prop:first cases intro}
In genus $0$ and $K3^{[2]}$-type Conjecture~\ref{conjecture:MC_intro} holds for $\beta = m \cdot \alpha$ where $\alpha$ is primitive whenever:
\begin{itemize}
\item $(\alpha, \alpha) < 0$, or
\item $(\alpha, \alpha) = 0$ and ($m=2$ or $N=1$), or
\item $(\alpha, \alpha) = 3/2$ and $m \in \{ 2,3,5 \}$.
\end{itemize}
\end{prop}

Ideally we would like to apply our methods to other families.
However, it is quite difficult to find appropriate $1$-parameter families.
They must be (a) a zero section of a homogeneous vector bundle on the GIT quotient of a vector space by a reductive group,
and (b) their singular fibers must have mild singularieties.
A promising candidate seemed to be a generic pencil of Iliev-Manivel fourfolds \cite{IM}
\[ \CX \subset \Gr(2,4) \times \Gr(2,4) \times \Gr(2,4) \times \p^1, \quad \pi : \CX \to \p^1 \]
but unfortunately the singular fibers appear to be too singular.\footnote{Another candidate
is the family of Fano varieties of Pfaffian cubics, $X \subset \Gr(4,6) \times \Gr(2,6)$, found by
Fatighenti and Mongardi \cite{FM}.}

\subsection{Debarre-Voisin fourfolds}
For the generic pencil $\pi : \CX \to \p^1$ of Debarre-Voisin fourfolds
the Noether-Lefschetz numbers have not yet been determined.
Instead we use the known cases of Conjecture~\ref{conjecture:MC_intro} and the mirror symmetry calculations for the total space $\CX$
to obtain constraints for the Noether-Lefschetz numbers.
By using the modularity of the generating series
of Noether-Lefschetz numbers due to Borcherds and McGraw \cite{Bor2, McGraw} we can then determine the full series.

Consider the generating series of Noether-Lefschetz numbers of $\pi$ as defined in Section~\ref{subsubsec:Borcherds form},
\[ \varphi(q) = \sum_{D \geq 0} q^{D/11} \NL^{\pi}(D) \]
where $D$ runs over all squares modulo $11$.

Define the weight $1,2,3$ modular forms
\begin{gather*}
E_{1}(\tau) = 1 + 2 \sum_{n \geq 1} q^n \sum_{d|n} \chi_p\left( \frac{n}{d} \right), \quad \ 
\Delta_{11}(\tau) = \eta(\tau)^2 \eta(11 \tau)^{2} \\
E_{3}(\tau) = \sum_{n \geq 1} q^n \sum_{d|n} d^{2} \chi_p\left( \frac{n}{d} \right)
\end{gather*}
where $\eta(\tau) = q^{1/24} \prod_{n \geq 1} (1-q^n)$ is the Dirichlet eta function,
$q = e^{2 \pi i \tau}$ and $\chi_{11}$ is the Dirichlet character given by the Legendre symbol $\left(\frac{ \cdot }{11} \right)$.
Consider the following weight $11$ modular forms for $\Gamma_0(11)$ and character $\chi_{11}$:
\begin{small}
\begin{align*}
\varphi_0(q) & = -5 \, E_{1}^{11} + 430 \, E_{1}^{8} E_{3} + \frac{5199920}{9} \, \Delta_{11}^{3} E_{1}^{5} - \frac{35407490}{27} \, \Delta_{11}^{4} E_{1}^{3} \\
&\ \ \  + \frac{49194440}{9} \, \Delta_{11}^{2} E_{1}^{4} E_{3} 
+ 248350 \, E_{1}^{5} E_{3}^{2} - \frac{596661440}{27} \, \Delta_{11}^{3} E_{1}^{2} E_{3} \\
&\ \ \  - \frac{306631760}{9} \, \Delta_{11} E_{1}^{3} E_{3}^{2} + \frac{51243500}{3} \, \Delta_{11}^{4} E_{3}
+ \frac{1331452540}{27} \, \Delta_{11}^{2} E_{1} E_{3}^{2}\\
&\ \ \  + \frac{349019440}{9} \, E_{1}^{2} E_{3}^{3} \\
& = -5 + 320q + 255420q^{2} + 14793440q^{3} + 262345260q^{4} + \ldots \\
\varphi_1(q) & = 
-5 \, E_{1}^{11} + 110 \, E_{1}^{8} E_{3} + \frac{722740}{3993} \, \Delta_{11}^{3} E_{1}^{5} - \frac{1805750}{3993} \, \Delta_{11}^{4} E_{1}^{3} \\
&\ \ \  - \frac{12660620}{11979} \Delta_{11}^{2} E_{1}^{4} E_{3} - 990 E_{1}^{5} E_{3}^{2} + \frac{118940}{363} \Delta_{11}^{5} E_{1} + \frac{5609180}{3993} \Delta_{11}^{3} E_{1}^{2} E_{3} \\
&\ \ \  + \frac{29208460}{11979} \, \Delta_{11} E_{1}^{3} E_{3}^{2} + \frac{3500}{33} \, \Delta_{11}^{4} E_{3} + \frac{2610980}{1089} \, E_{1}^{2} E_{3}^{3} \\
& = -5 + 320q^{11} + 990q^{12} + 5500q^{14} + 11440q^{15} + \ldots
\end{align*}
\end{small}

\begin{thm} \label{thm:DV NL series}
Let $\pi : \CX \to \p^1$ be a generic pencil of Debarre-Voisin fourfolds. Then
the generating series of its Noether-Lefschetz numbers is:
\begin{align*}
\varphi(q^{11}) = \varphi_0(q^{11}) + \varphi_1(q) = 
 -10 + 640 q^{11} + 990 q^{12} + 5500 q^{14} + 11440 q^{15} \\ + 21450 q^{16} 
+ 198770 q^{20} + 510840 q^{22} + \ldots  \\
\end{align*}
\end{thm}

Debarre-Voisin varieties are parametrized by a $20$-dimensional projective irreducible GIT quotient
\[ \CM_{\text{DV}} = \p( \wedge^3 V_{10}^{\vee} )// \mathrm{SL}(V_{10}) \]
where $V_{10}$ is a vector space of dimension $10$. 
The period map from this moduli space to the moduli space of holomorphic-symplectic varieties
\[ p : \CM_{\text{DV}} \dashrightarrow \CM_{H} \]
is birational \cite{OGrady} and regular on the open locus corresponding to smooth Debarre-Voisin varieties of dimension $4$.
When passing to the Baily-Borel compactification $\overline{\CM_H}$ this birational map can be resolved.
An \emph{HLS divisor} (for Hassett-Looijenga-Shah) in $\CM_H$ is the image of an exceptional divisor under this resolved map.
These divisors reflect a difference between the GIT and the Baily-Borel compactification
as they parametrize holomorphic-symplectic fourfolds of the same polarization type as a Debarre-Voisin fourfold,
but for which the generic member is not a Debarre-Voisin fourfold.

Let $e$ be a square modulo $11$ and let
\[ \CC_{2e} \subset \CM_{H} \]
be the Noether-Lefschetz divisor of the first type of discriminant $e$,
see Section~\ref{subsec:NL divisors first type} for the precise definition.
Observe that there is a natural gap\footnote{In fact, the gap determines the modular form (viewed as a vector-valued modular form) up to a constant.} in the modular form $\varphi(q)$:
\[ \varphi(q^{11}) =  -10 + \underbrace{ 0 \cdot q^1 + 0 \cdot q^{3} + \ldots + 0 \cdot q^{9} }_{\text{gap}} + 640 q^{11} + 990 q^{12} + 5500 q^{14} + \ldots \]
Translating from Heegner divisors to the irreducible divisors $\CC_{2e}$,
this gap yields the following:

\begin{cor} \label{corHLS}
The divisors $\CC_{2}, \CC_{6}, \CC_{8}, \CC_{10}, \CC_{18}$ are HLS divisors of the moduli space of Debarre-Voisin fourfolds.
The divisor $\CC_{30}$ is not HLS.
\end{cor}

The statement that $\CC_{2}, \CC_{6}, \CC_{10}, \CC_{18}$ are HLS is the main result of \cite{DHOV}.
The argument here gives an independent and mostly formal proof of the main result of \cite{DHOV}.
The only geometric input lies in describing in understanding the geometry of the singular fibers (a result of J.~Song, see Appendix~\ref{appendix:Song}).
The result that $\CC_{30}$ is not a HLS divisor answers a question of \cite{DHOV}.
The fact that $\CC_{8}$ is HLS seems to be new, and it would be interesting to understand the geometry of these loci, as done in \cite{DHOV} for the other cases.
In principle, Theorem~\ref{thm:DV NL series} can be used to show that the divisors listed in Corollary~\ref{corHLS}
are the only Noether-Lefschetz divisors $\CC_{2e}$ which are HLS. 

\begin{rmk}
After the first version of this paper appeared online,
I learned that Theorem~\ref{thm:DV NL series} and Corollary~\ref{corHLS}
was independently obtained by Calla Tschanz based on the results of \cite{DHOV}.
\end{rmk}

\subsection{Convention}
If $\gamma \in H^{i}(X,\BQ)$ is a cohomology class, we write $\deg(\gamma) = i/2$ for the complex cohomological degree of $\gamma$.
For $X$ holomorphic-symplectic we identify $\Pic(X)$ with its image in $H^2(X,\BZ)$ under the map taking the first Chern class.
Let $\Hilb_n(S)$ be the Hilbert scheme of points of a K3 surface.
Given $\alpha \in H^{\ast}(S,\BQ)$ and $i > 0$ we let
\[ \Fq_i(\alpha) : H^{\ast}(\Hilb_n(S)) \to H^{\ast}(\Hilb_{n+i}(S)) \]
be the $i$-th Nakajima operator \cite{Nak} given by adding a $i$-fat point on a cycle with class $\alpha$;
we use the convention of \cite{OLLV}.
We write $A \in H_2(\Hilb_n(S))$ for the class of a generic fiber of the singular locus of the Hilbert Chow morphism,
and we let $-2 \delta$ be the class of the diagonal. We identify
\[ H^2(\Hilb_n(S)) \equiv H^2(S,\BZ) \oplus \BZ \delta, \quad H_2(\Hilb_n(S)) \equiv H_2(S,\BZ) \oplus \BZ A \]
using the Nakajima operators \cite{HilbK3}.

\subsection{Subsequent work}
In \cite{QuasiK3} the main conjecture of this paper (Conjecture~\ref{conjecture:MC_intro})
is proven for all $K3^{[n]}$ in genus $0$ and for $N \leq 3$ markings.

\subsection{Acknowledgements}
The idea to use Noether-Lefschetz theory for the Gromov-Witten theory of $K3^{[2]}$-type varieties
is due to E.~Scheidegger and quite old \cite{STalk}.
I also owe a great debt to the beautiful paper on Noether-Lefschetz theory by D.~Maulik and R.~Pandharipande \cite{MP}.
I further thank T.~Beckmann, J.~Bryan, T.~H.~Buelles, O.~Debarre, E.~Markman, G.~Mongardi, R.~Mboro, and J.~Song for useful comments,
and the referees for a careful reading and helpful remarks.
The author was funded by the Deutsche Forschungsgemeinschaft (DFG) -- OB 512/1-1.

\section{The monodromy in $K3^{[n]}$-type}
\subsection{Overview}
Let $X$ be a (irreducible) holomorphic-symplectic variety.
The lattice $H^2(X,\BZ)$ is equipped with the integral and non-degenerate Beauville-Bogomolov-Fujiki quadratic form.
We will also equip $H^{\ast}(X,\BZ)$ with the usual Poincar\'e pairing.
Both pairings are extended to the $\BC$-valued cohomology groups by linearity.
Let $\Mon(X)$ be the subgroup of $O(H^{\ast}(X,\BZ))$ generated by all monodromy operators,
and let $\Mon^2(X)$ be its image in $O(H^2(X,\BZ))$.

The goal of this section is to describe the monodromy group in the case that $X$ is of $K3^{[n]}$-type,
and we will assume so from now on.
The main references for the sections are Markman's papers \cite{MarkmanSurvey, Markman}.

\subsection{Monodromy}
Let $X$ be of $K3^{[n]}$-type.
By work of Markman \cite[Thm.1.3]{MarkmanSurvey}, \cite[Lemma 2.1]{Markman3} we have that
\begin{equation} \Mon(X) \cong \Mon^2(X) = \widetilde{O}^+(H^2(X,\BZ)) \label{abc2} \end{equation}
where the first isomorphism is the restriction map and
$\widetilde{O}^+(H^2(X,\BZ))$
is the subgroup of $O(H^2(X,\BZ))$ of orientation preserving lattice automorphisms which act by $\pm 1$ on the discriminant.\footnote{Let
$\CC = \{ x \in H^2(X,\BR) | \langle x, x \rangle > 0 \}$ be the positive cone. Then $\CC$ is homotopy equivalent to $S^2$. An automorphism is orientation preserving
if it acts by $+1$ on $H^2(\CC) = \BZ$.}
The first isomorpism implies that any parallel transport operator $H^{\ast}(X_1, \BZ) \to H^{\ast}(X_2, \BZ)$
between two $K3^{[n]}$-type varieties is uniquely determined by its restriction to $H^2(X_1,\BZ)$.

If $g \in \Mon^2(X)$, we let $\tau(g) \in \{ \pm 1 \}$ be the sign by which $g$ acts on the discriminant lattice. 
This defines a character
\[ \tau : \Mon^2(X) \to \BZ_2. \]

\subsection{Zariski closure} \label{subsec:zariski closure}
By \cite[Lemma 4.11]{Markman} if $n \geq 3$ the Zariski closure of $\Mon(X)$ in $O(H^{\ast}(X,\BC))$ is $O(H^2(X,\BC)) \times \BZ_2$.
The inclusion yields the representation
\begin{equation} \rho : O(H^2(X,\BC)) \times \BZ_2 \to O(H^{\ast}(X,\BC)) \label{repn} \end{equation}
which acts by degree-preserving orthogonal ring isomorphism.
There is a natural embedding
\[ \widetilde{O}^+(H^2(X,\BZ)) \to O(H^2(X,\BC)) \times \BZ_2,\, g \mapsto (g, \tau(g)) \]
under which $\rho$ restricts to the monodromy representation.
In case $n \in \{ 1 ,2 \}$ the Zariski closure is $O(H^2(X,\BC))$. In this case, we define 
the representation \eqref{repn} by letting it act through $O(H^2(X,\BC))$.

The representation $\rho$ is determined by the following properties:

\vspace{5pt}
\noindent
\textbf{Property 0.} For any $(g, \tau) \in O(H^2(X,\BC)) \times \BZ_2$ we have 
\[ \rho(g,\tau)|_{H^2(X,\BC)} = g. \]
\textbf{Property 1.} The restriction of $\rho$ to $\SO(H^2(X,\BC)) \times \{ 0 \}$ is the
integrated action of the Looijenga-Lunts-Verbitsky algebra \cite{LL, V}.\\[5pt]
\textbf{Property 2.} We have 
\[ \rho(1, -1) = D \circ \rho(-\id_{H^2(X,\BC)}, 1), \]
where
$D$ acts on $H^{2i}(X,\BC)$ by multiplication by $(-1)^i$.\\[5pt]
\textbf{Property 3.} Assume that $X = \Hilb_n(S)$ and identify $H^2(X,\BZ)$ with $H^2(S,\BZ) \oplus \BZ \delta$.
Then the restriction of $\rho$ to $O(H^2(X,\BC))_{\delta} \times 1$
(identified naturally with $O(H^2(S,\BC))$)
is the Zariski closure of the induced action on the Hilbert scheme by the monodromy representation of $S$.

In particular, the action is equivariant with respect to the Nakajima operators:
For $g \in O(H^2(X,\BC))_{\delta}$ let $\tilde{g} = g|_{H^2(S,\BC))} \oplus \id_{H^0(S,\BZ) \oplus H^4(S,\BZ)}$. Then
\[ \rho(g, 1) \left( \prod_i \Fq_{k_i}(\alpha_i) 1 \right) = \prod_i \Fq_{k_i}( \tilde{g} \alpha_i ) 1. \]

\vspace{5pt}
\noindent \textbf{Property 4.} Let $P_{\psi} : H^{\ast}(X_1, \BZ) \to H^{\ast}(X_2, \BZ)$ be a parallel transport operator with $\psi = P_{\psi}|_{H^2(X_1,\BZ)}$.
Then
\[ P_{\psi}^{-1} \circ \rho(g, \tau) \circ P_{\psi} = \rho( \psi^{-1} g \psi, \tau). \]

Property 1 follows by \cite[Lemma 4.13]{Markman}.
The other properties also follow from the results of \cite{Markman}.
Properties 1-3 determine the action $\rho$ completely
in the Hilbert scheme case. Moreover by \cite{OLLV}
this description is explicit in the Nakajima basis.
The last condition extends this presentation then to arbitrary $X$.
The parallel transport operator between different moduli spaces of stable sheaves can also be described more explicitly \cite{Markman}.

\subsection{Parallel transport}
Let 
\[ \Lambda = E_8(-1)^{\oplus 2} \oplus U^4 \]
be the Mukai lattice.
For $n \geq 2$, any holomorphic-symplectic variety of $K3[n]$-type is
equipped with a canonical choice of a primitive embedding
\[ \iota_X : H^2(X,\BZ) \to \Lambda \]
unique up to composition by an element by $O(\Lambda)$, see \cite[Cor.9.4]{MarkmanSurvey}.

\begin{thm}(\cite[Thm.9.8]{MarkmanSurvey}) \label{thm:pt operators}
An isometry $\psi : H^2(X_1, \BZ) \to H^2(X_1, \BZ)$
is the restriction of a parallel transport operator if and only of if it is orientation preserving
and there exists an $\eta \in O(\Lambda)$ such that
\[ \eta \circ \iota_{X_1} = \iota_{X_2} \circ \psi, \]
\end{thm}
Orientation preserving is here defined with respect to the canonical choice of orientation of the positive cone of $X_i$
given by the real and imaginary part of the symplectic form and a K\"ahler class.
If $X_1 = X_2$ the theorem reduces to the second isomorphism in \eqref{abc2}.

\subsection{Curve classes} \label{subsec:curve classes}
By Eichler's criterion \cite[Lemma 3.5]{GHS}, Theorem~\ref{thm:pt operators} yields a complete set of deformation invariants of curve classes in $K3^{[n]}$-type.
To state the result we need the following constriction:

The orthogonal complement
\[ L = \iota_X(H^2(X,\BZ))^{\perp} \subset \Lambda \]
is isomorphic to the lattice $\BZ$ with intersection form $(2n-2)$.
Let $v \in L$ be a generator and consider the isomorphism of abelian groups
\begin{equation} L^{\vee} / L \xrightarrow{\cong} \BZ / (2n-2)\BZ  \label{sdffd} \end{equation}
determined by sending $v/(2n-2)$ to the residue class of $1$.
Since the generators of $L$ are $\pm v$, \eqref{sdffd} is canonical up to multiplication by $1$.

Since $\Lambda$ is unimodular there exists a natural isomorphism (\cite[Sec.14]{HuybrechtsLecturesOnK3})
\[ H^2(X,\BZ)^{\vee} / H^2(X,\BZ) \xrightarrow{\cong} L^{\vee}/L. \]
If we use Poincar\'e duality to identify $H_2(X,\BZ)$ with $H^2(X,\BZ)^{\vee}$
this yields the \emph{residue map}
\[ r_X:H_2(X,\BZ)/H^2(X,\BZ) \xrightarrow{\cong} L^{\vee}/L \xrightarrow{\cong} \BZ/(2n-2) \BZ. \]
The map depends on the choice of the generator $v$ and hence is unique up to multiplication by $\pm 1$.

\begin{defn}
The \emph{residue set} of a class $\beta \in H_2(X,\BZ)$ is defined by
\begin{equation*} \label{res_set} \pm [\beta] = \{ \pm r_X([\beta]) \} \subset \BZ / (2n-2) \BZ \end{equation*}
if $n \geq 2$, and by $\pm [\beta]=0$ otherwise.
\end{defn}

Note that since $r_X$ is canonical up to sign, the residue set is independent of the choice of map $r_X$. 

\begin{rmk}
(i) Since parallel transport operators respect the embedding $i_X$ up to composing with an isomorphism of $\Lambda$,
the residue set $[\beta]$ is preserved under deformation. (This is also reflected in the fact, that the monodromy acts by $\pm 1$ on the discriminant.)\\
(ii) In the case of the Hilbert scheme $X = \Hilb_n(S)$ of a K3 surface, let $A \in H_2(X)$ be the class of an exceptional curve,
that is the class of a fiber of the Hilbert-Chow morphism $\Hilb_n(S) \to \mathrm{Sym}^n(S)$ over a generic point in the singular locus.
We have a natural identification
\[ H_2(X,\BZ) = H_2(S,\BZ) \oplus \BZ A. \]
The morphism $r_X$ then sends (up to sign) the class $[A]$ to $1 \in \BZ_{2n-2}$. \\
(iii) In other words, we could have defined the residue class also by first deforming to the Hilbert scheme, and then taking the coefficient of $A$ modulo $2n-2$.
This is usually the practical way to compute the residue class. \qed
\end{rmk}

Let $\beta \in H_2(X,\BZ)$ be a class.
The class $\beta$ has then the following deformation invariants:
\begin{enumerate}
\item[(i)] the divisibility $\div(\beta)$ in $H_2(X,\BZ)$,
\item[(ii)] the Beauville--Bogomolov norm $(\beta, \beta) \in \BQ$, and
\item[(iii)] the residue set $\pm \left[ \frac{\beta}{\div(\beta)} \right] \in \BZ / (2n-2) \BZ$.
\end{enumerate}
The global Torelli theorem and Eichler's criterion yields the following:
\begin{cor} \label{cor:curve deformations invariance}
Two pairs $(X,\beta)$ and $(X',\beta')$ of a $K3^{[n]}$-type variety and a class in $H_2(X,\BZ)$ which pairs positively with a K\"ahler class
are deformation equivalent if and only if the invariants (i-iii) agree.
\end{cor}

\begin{rmk} \label{rmkHodge}
By the global Torelli theorem, if $\beta$ and $\beta'$ are both of Hodge type,
the deformation in the corollary can be choosen such that the curve class stays of Hodge type.
\end{rmk}

\subsection{Lifts of isometries of $H^2$} \label{Subsec:Parallel Transport lifts}
Let $X_1, X_2$ be of $K3^{[n]}$-type,
and let
\[ g : H^2(X_1,\BC) \to H^2(X_2,\BC) \]
be an isometry.
An operator $\widetilde{g} : H^{\ast}(X_1,\BC) \to H^{\ast}(X_2,\BC)$
is a \emph{parallel transport lift} of $g$ if it is of the form
\[ \widetilde{g} = \rho(g \circ \psi^{-1}, 1) \circ P_{\psi} \]
for a parallel transport operator
$P_{\psi} : H^{\ast}(X_1, \BZ) \to H^{\ast}(X_2,\BZ)$ with restriction $\psi = P|_{H^2(X_1,\BZ)}$.
In particular, any parallel transport lift is a degree-preserving orthogonal ring isomorphism.

Recall from Section~\ref{subsec:zariski closure}, Property 2, the operator
\[ \widetilde{D} = \rho(\id, -1) = D \circ \rho(-\id_{H^2(X_2,\BZ)}, 1). \]

\begin{lemma} A parallel transport lift of $g$ is unique up to composition by $\widetilde{D}$.
\end{lemma}
\begin{proof}
Consider two parallel transport lifts of $g$,
\[ \widetilde{g_i} =  \gamma(g \circ \psi_i^{-1}) \circ P_{\psi_i}, \quad i=1,2 \]
for parallel transport operators $P_{\psi_1}, P_{\psi_2}$.
We will show that
\[ \widetilde{g}_1 = \widetilde{g}_2 \quad \text{or} \quad \widetilde{g}_1 = \widetilde{D} \circ \widetilde{g}_2. \]

Let $\gamma(h) = \rho(h, 1)$. If $\tau(\psi_1 \circ \psi_2^{-1}) = 1$ then 
\[
\gamma(g \circ \psi_1^{-1}) \circ P_{\psi_1} \circ P_{\psi_2}^{-1} = \gamma(g \psi_1^{-1}) \circ \gamma(\psi_1 \circ \psi_2^{-1}) = \gamma( g \psi_2^{-1}).
\]
If $\tau(\psi_1 \circ \psi_2^{-1}) = -1$ then 
\begin{multline*}
\gamma(g \circ \psi_1^{-1}) \circ P_{\psi_1} \circ P_{\psi_2}^{-1} = \gamma(g \psi_1^{-1}) \circ D \circ \gamma(-\psi_1 \circ \psi_2^{-1})
= \widetilde{D} \circ  \gamma( g \psi_2^{-1})
\end{multline*}
\end{proof}

\section{The multiple cover conjecture} \label{Multiple cover rule}
\subsection{Overview}
Let $X$ be a variety of $K3^{[n]}$-type and let
$\beta \in H_2(X,\BZ)$ be an effective curve class.
The moduli space $\Mbar_{g,N}(X,\beta)$ of $N$-marked genus $g$ stable maps to $X$ in class $\beta$ carries a reduced virtual fundamental class
$[ \Mbar_{g,N}(X,\beta) ]^{\text{red}}$
of dimension $(2n-3)(1-g)+N+1$.
Gromov-Witten invariants of $X$ are defined by pairing with this class:
\begin{equation}
\label{GWinvts}
\blangle \alpha ; \gamma_1, \ldots, \gamma_n \brangle^X_{g,\beta} := \int_{ [ \Mbar_{g,n}(X,\beta) ]^{\text{red}} } \pi^{\ast}(\alpha) \cup \prod_i \ev_i^{\ast}(\gamma_i),
\end{equation}
where $\ev_i : \Mbar_{g,n}(X,\beta) \to X$ are the evaluation maps, $\tau : \Mbar_{g,n}(X,\beta) \to \Mbar_{g,n}$ is the forgetful map,
and $\alpha \in H^{\ast}(\Mbar_{g,n})$ is a \emph{tautological class} \cite{FP}.

In this section we state a conjecture to express the invariants \eqref{GWinvts} for $\beta$ an arbitrary class
in terms of invariants where $\beta$ is primitive.

\subsection{Invariance} \label{subsec:invariance}
Let $\beta \in H_2(X,\BZ)$ be an effective curve class,
and let 
\[ \widetilde{O}^{+}(H^2(X,\BZ))_{\beta} \subset \widetilde{O}^+(H^2(X,\BZ)) \]
be the subgroup fixing $\beta$ (either via the monodromy representation or equivalently,
via the dual action on $H^2(X,\BZ)^{\vee} \cong H_2(X,\BZ)$ under the Beauville-Bogomolov form).
Applying Remark~\ref{rmkHodge} 
and the deformation invariance of the reduced Gromov-Witten invariants, we find that
\[
\Blangle \alpha\, ;\, \gamma_1, \ldots, \gamma_n \Brangle^X_{g,\beta}
=
\Blangle \alpha\, ;\, \mu(\varphi) \gamma_1,\ \ldots,\ \mu(\varphi) \gamma_n \Brangle^X_{g,\beta}
\]
for all $\varphi \in \widetilde{O}^{+}(H^2(X,\BZ))_{\beta}$, where we have used
\[ \mu : \widetilde{O}^{+}(H^2(X,\BZ)) \to O(H^{\ast}(X,\BZ)) \]
to denote the monodromy representation (defined by the isomorphism \eqref{abc2}
composed with the inclusion $\Mon(X) \subset O(H^{\ast}(X,\BZ))$).

The image of $\widetilde{O}^{+}(H^2(X,\BZ))_{\beta}$ is Zariski dense in
\begin{align*}
G_{\beta} 
& = (O(H^2(X,\BC)) \times \BZ/2)_{\beta} \\
& := \{ g \in O(H^2(X,\BC)) \times \BZ/2 \BZ\ |\ \rho(g) \beta = \beta \}.
\end{align*} 
It follows that for all $g \in G_{\beta}$ we have
\begin{equation}
\label{invariance}
\Blangle \alpha\, ;\, \gamma_1, \ldots, \gamma_n \Brangle^X_{g,\beta}
=
\Blangle \alpha\, ;\, \rho(g) \gamma_1,\ \ldots,\ \rho(g) \gamma_n \Brangle^X_{g,\beta}.
\end{equation}
Equivalently, the pushforward of the reduced virtual class lies in the invariant part
of the diagonal $G_{\beta}$ action:
\begin{equation*} \ev_{\ast} \left( \tau^{\ast}(\alpha) \cap [\Mbar_{g,n}(X,\beta)]^{\text{red}} \right) \in H^{\ast}(X^n, \BQ)^{G_{\beta}}. \label{invariance} \end{equation*}

The representation $\rho$ restricted to $\{ \id \} \times \BZ_2$ acts trivially on $H^2(X,\BC)$.
Hence for $X$ a K3 surface, we obtain the invariance of the Gromov-Witten class 
under the group $O(H^2(X,\BC))_{\beta}$.
This matches what was conjectured in \cite[Conj.C1]{K3xE} and then proven along the above lines in \cite{Bue1}.

\subsection{Multiple-cover conjecture} \label{subsec:mc conjecture}
The main conjecture is the following: Let $\beta \in H_2(X,\BZ)$ be any effective curve class.
For every divisor $k | \beta$,
let $X_k$ be a variety of $K3^{[n]}$ type and let
\[ \varphi_k : H^{2}(X,\BR) \to H^{2}(X_k,\BR) \]
be a real isometry such that:
\begin{itemize}
\item $\varphi_k(\beta/k)$ is a primitive curve class
\item $\pm \left[ \varphi_k(\beta/k) \right] = \pm [\beta/k]$.
\end{itemize}
We extend $\varphi_k$ as a \emph{parallel transport lift} (Section~\ref{Subsec:Parallel Transport lifts}) to the full cohomology:
\[ \varphi_k : H^{\ast}(X, \BR) \to H^{\ast}(X_k, \BR), \]
By Section~\ref{Subsec:Ex:Hilb Elliptic Curves} below, pairs $(X_k, \varphi_k)$ satisfying these properties can always be found.

\begin{conjecture} \label{conjecture:MC}
For any effective curve class $\beta \in H_2(X,\BZ)$ we have
\begin{multline*}
\Big\langle \alpha ; \gamma_1, \ldots, \gamma_N \Big\rangle^X_{g,\beta} \\
=
\sum_{k | \beta} 
k^{3g-3+N - \deg(\alpha)}
(-1)^{[\beta] + [\beta/k]}
\Big\langle \alpha ; \varphi_k(\gamma_1), \ldots, \varphi_k(\gamma_N) \Big\rangle^{X_k}_{g,\varphi_k(\beta/k)}.
\end{multline*}
\end{conjecture}
\vspace{8pt}

The invariance property discussed in Section \ref{invariance} and Property 4 of Section~\ref{subsec:zariski closure}
imply that the right hand side of the conjecture is independent of the choice of $(X_k,\varphi_k)$.
Using that $\sum_{k|a} \mu(k) = \delta_{a1}$, Conjecture~\ref{conjecture:MC} is also seen to be equivalent to Conjecture~\ref{conjecture:MC_intro} of the introduction.

The reduced Gromov--Witten invariants of $X$ can only be non-zero if the dimension constraint
\[ (\dim X - 3)(1-g) + N + 1 = \deg(\alpha) + \sum_i \deg(\gamma_i) \]
is satisfied. Hence the conjecture can also be rewritten as:
\begin{multline*}
\Blangle \alpha ; \gamma_1, \ldots, \gamma_N \Brangle^X_{g,\beta} \\
=
\sum_{k | \beta} 
k^{\dim(X)(g-1) - 1 + \sum_i \deg(\gamma_i)}
(-1)^{[\beta] + [\beta/k]}
\Blangle \alpha ; \varphi_k(\gamma_1), \ldots, \varphi_k(\gamma_N) \Brangle^X_{g,\varphi_k(\beta/k)}.
\end{multline*}
Since for K3 surfaces the residue always vanishes, Conjecture~\ref{conjecture:MC} specializes
for K3 surfaces to the conjecture made in \cite[Conj C2]{K3xE}.

\begin{rmk} The condition that we ask of the residue, i.e.
$\pm \left[ \varphi_k(\beta/k) \right] = \pm [\beta/k]$ is necessary for the conjecture to hold.
For example consider $X$ of $K3^{[5]}$-type and two primitive classes $\beta_1, \beta_2$ with
$(\beta_1 \cdot \beta_1) = (\beta_2 \cdot \beta_2) = 16$ but $[\beta_1]=0$ and $[\beta_2] = 4$ in $\BZ/8 \BZ$.
Then by \cite{HilbK3} one has 
\[ \ev_{\ast}[\Mbar_{0,1}(X,\beta_1)]^{\text{red}} = 1464 \beta_1^{\vee}, \quad  \ev_{\ast}[\Mbar_{0,1}(X,\beta_2)]^{\text{red}} = 480 \beta_2^{\vee} \]
so an isometry taking $\beta_1$ to $\beta_2$ does not preserve Gromov-Witten invariants.
\end{rmk}

\subsection{Uniruled divisors}
An uniruled divisor $D \subset X$ which is swept out by a rational curve in class $\beta$
is a component of the image of $\ev:\Mbar_{0,1}(X,\beta) \to X$.
The virtual class of these uniruled divisors is given by the pushforward
$\ev_{\ast} [\Mbar_{0,1}(X,\beta)]^{\text{red}}$.
For $\beta$ primitive and $(X,\beta)$ very general, the virtual class is closely related to the actual class \cite{OSY}.

By monodromy invariance (e.g. \cite[Section 2.6]{OSY}) there exists $N_{\beta} \in \BQ$ such that
\begin{equation} \ev_{\ast} [\Mbar_{0,1}(X,\beta)]^{\text{red}} = N_{\beta} \cdot h \label{uniruled case} \end{equation}
where $h = ( \beta, - ) \in H_2(X,\BQ)^{\vee} \cong H^2(X,\BQ)$ is the dual of $\beta$
with respect to the Beauville-Bogomolov-Fujiki form.
Since by deformation invariance $N_{\beta}$ only depends on the divisibility $m = \div(\beta)$, the square $s = (\beta, \beta)$ and the residue $r = [\beta / \div(\beta)]$
we write 
\[ N_{\beta} = N_{m,s,r}. \]
Conjecture~\ref{conjecture:MC} then says that
\begin{equation} \label{conj_for_uniruled}
N_{m,s,r} = \sum_{k | \beta} \frac{1}{k^3} (-1)^{mr + \frac{m}{k} r} N_{1, \frac{s}{k^2}, \frac{mr}{k}}
\end{equation}
The primitive numbers $N_{1,s,r}$ have been determined in \cite{HilbK3}.

%

\begin{example} \label{Ex1}
Let $X = \Hilb_2(S)$, and let $A \in H_2(X,\BZ)$ be the class of the exceptional curve.
We have
\[ \ev_{\ast} [\Mbar_{0,1}(X,\beta)]^{\text{red}} = \Delta_{\Hilb_2(S)} = -2 \delta. \]
Since $A^{\vee} = -\frac{1}{2} \delta$ we see $N_{1,-\frac{1}{2},1} = 4$.
The multiple cover formula then predicts
that for even $\ell \in \BZ_{\geq 1}$ we have
\[ N_{\ell A} = -\frac{1}{\ell^3} N_{1,-\frac{1}{2},1} + \frac{1}{(\ell/2)^3} N_{1,-2,0} = \frac{-4}{\ell^3} + \frac{8}{\ell^3} = \frac{4}{\ell^3} \]
where we have used $N_{1,-2,0} = 1$. This matches the degree-scaling property discussed in \cite{OkPandHilbC2, MO}. \qed
\end{example}
%
%

\subsection{Fourfolds}
We consider genus $0$ Gromov-Witten invariants of a variety $X$ of $K3^{[2]}$-type.
By dimension considerations 
all genus $0$ Gromov-Witten invariants are determined by the $2$-point class:
\[ \ev_{\ast} [\Mbar_{0,2}(X,\beta)]^{\text{red}} \in H^{\ast}(X \times X). \]

Following the arguments of \cite[Sec.2]{OSY}\footnote{One 
uses that the class is monodromy invariant and a Lagrangian correspondence,
and that for very general $(X,\beta)$ the image of the Hodge classes under this correspondence is annihilated by the symplectic form, see \cite[Sec.1.3]{OSY}. The reference treats only the case of primitive $\beta$, but the imprimitive case follows likewise with minor modifications.},
for every effective $\beta \in H_2(X,\BZ)$
there exist constants\footnote{The constant $N_{\beta}$ appearing in \eqref{uniruled case} is equal to $G_{\beta}$ in the $K3^{[2]}$-case.}
$F_{\beta}, G_{\beta} \in \BQ$
such that:
\begin{align*}
& \ev_{\ast}[\Mbar_{0,2}(X,\beta)]^{\text{red}} = 
 F_{\beta} (h^2 \otimes h^2) + 
G_{\beta} \big(h \otimes \beta + \beta \otimes h + (h \otimes h) \cdot c_{BB} \big)  \\
& \ + \left( -\frac{1}{30} (h^2 \otimes c_2 + c_2 \otimes h^2) + \frac{1}{900} (\beta, \beta) c_2 \otimes c_2 \right) (G_{\beta} + (\beta, \beta) F_{\beta} )
\end{align*}
where 
\begin{itemize}
\item $h = ( \beta, - ) \in H^2(X,\BQ)$ is the dual of the curve class,
\item $c_2 = c_2(X)$ is the second Chern class, and
\item $c_{BB} \in H^2(X) \otimes H^2(X)$ is the inverse of the Beauville-Bogomolov Fujiki form.
\end{itemize}

In $K3^{[2]}$-type
the residue $r = [\beta]$ of a curve class is determined by $s = (\beta,\beta)$ via $r = 2 s \text{ mod } 2$.
So we can write
$F_{\beta} = F_{m,s}$ and $G_{\beta} = G_{m,s}$. 
The multiple cover conjecture for $K3^{[2]}$-type in genus $0$ is then equivalent to:
\begin{gather*}
F_{m,s} = \sum_{k|m} \frac{1}{k^5} (-1)^{2 (s+s/k^2)} F_{1,\frac{s}{k^2}} \\
G_{m,s} = \sum_{k|m} \frac{1}{k^3} (-1)^{2 (s+s/k^2)} G_{1,\frac{s}{k^2}}.
\end{gather*}
We can also define 
\begin{gather*}
f_{m,s} = \sum_{k|m} \frac{\mu(k)}{k^5} (-1)^{2(s + s/k^2)} F_{\frac{m}{k},\frac{s}{k^2}} \\
g_{m,s} = \sum_{k|m} \frac{\mu(k)}{k^3} (-1)^{2(s + s/k^2)} G_{\frac{m}{k},\frac{s}{k^2}}.
\end{gather*}
and arrive at:

\begin{lemma}
Conjecture~\ref{conjecture:MC} in $K3^{[2]}$-type and genus $0$ is equivalent to:
\[ \forall m,s: \ f_{m,s} = f_{1,s}, \quad g_{m,s} = g_{1,s}. \]
\end{lemma}

The first few cases are known:

\begin{prop} \label{prop:first cases}
Conjecture~\ref{conjecture:MC} in $K3^{[2]}$-type and genus $0$ holds for all classes $\beta$
such that (i) $(\beta, \beta) < 0$ or (ii) $(\beta, \beta) = 0$ and $N=1$.
\end{prop}
\begin{proof}
The case $(\beta, \beta) < 0$ follows from \cite{OkPandHilbC2, MO}.
In case $(\beta, \beta) = 0$ the series $g_{m,s}$ is determined by intersecting the $1$-pointed class with a curve,
and then use the methods of \cite{HilbK3} to reduce to $\Hilb_2(\p^1 \times E)$. The resulting series is evaluated by $T$ in \cite[Thm.9]{HilbK3}.
\end{proof}

For later use we will also need to following expression for $1$-pointed descendence invariants:
\begin{align*}
\ev_{\ast} [\Mbar_{0,1}(X,\beta) ] & = G_{\beta} \beta^{\vee} \\
\ev_{\ast} \left( \psi_1 \cdot [\Mbar_{0,1}(X,\beta) ] \right) & = 2 F_{\beta} h^2 - \frac{1}{15} \left( G_{\beta} + (\beta, \beta) F_{\beta} \right) c_2(X) \\
\ev_{\ast} \left( \psi_1^2 \cdot [\Mbar_{0,1}(X,\beta) ] \right) & = -12 F_{\beta} \\
\ev_{\ast} \left( \psi_1^3 \cdot [\Mbar_{0,1}(X,\beta) ] \right) & = 24 F_{\beta}
\end{align*}
This follows by monodromy invariance and topological recursions.
In particular, to check the multiple cover formula in $K3^{[2]}$-type and genus $0$ it is enough to consider $1$-point descendent invariants.

\begin{rmk} \label{rmk:primitive evaluations}
For convenience we recall the evaluation of $f_{1,m,s}$ and $g_{1,m,s}$. Let
\begin{gather*}
\vartheta(q) = \sum_{n \in \BZ} q^{n^2}, \quad 
\alpha(q) = \sum_{\substack{\text{odd }n > 0 \\ d|n}} d q^n, \\
G_2(q) = -\frac{1}{24} + \sum_{n \geq 1} \sum_{d|n} d q^n, \quad 
\Delta(q) = q \prod_{n \geq 1} (1-q^n)^{24}
\end{gather*}
Then by \cite{HilbK3} one has: 
\begin{align*}
\sum_{s} f_{1,s} (-q)^{4s} &= \frac{1}{4} \frac{-1}{\vartheta(q) \alpha(q) \Delta(q^4)} \\
\sum_{s} g_{1,s} (-q)^{4s} &= \frac{1}{12} \frac{\vartheta^{4}(q) + 4 \alpha(q) + 24 G_2(q^4)}{\vartheta(q) \alpha(q) \Delta(q^4)}.
\end{align*}
\end{rmk}

\subsection{Hilbert schemes of elliptic K3 surfaces} \label{Subsec:Ex:Hilb Elliptic Curves}
Let 
\[ \pi : S \to \p^1 \]
be an elliptic K3 surface with a section. Let $B, F \in H^2(S,\BZ)$ be the class of the section
and a fiber respectively, and let 
\[ W = B+F. \]
We consider the Hilbert scheme $X = \Hilb_n S$ and the generating series of Gromov-Witten invariants
\[
F_{g,m}(\alpha; \gamma_1, \ldots, \gamma_N) = \sum_{d = -m}^{\infty} \sum_{r \in \BZ} \left\langle \alpha ; \gamma_1, \ldots, \gamma_N \right\rangle^{\Hilb_n S}_{g, mW+dF+rA} q^d (-p)^r.
\]

We state a characterization of the multiple cover formula:

Consider the basis of $H^{\ast}(S,\BR)$ defined by 
\[ \CB = \{ 1, \pt, W, F, e_3, \ldots, e_{22} \}, \]
where $\pt$ is the class of a point and $e_3, \ldots, e_{22}$ is a basis of $\BQ\langle W,F \rangle^{\perp}$ in $H^2(S,\BR)$.
An element $\gamma \in H^{\ast}(X,\BR)$ is in the Nakajima basis with respect to $\CB$ if
it is of the form
\[ \gamma = \prod_i \Fq_{k_i}(\alpha_i)1, \quad \alpha_i \in \CB. \]
Let $w(\gamma)$ and $f(\gamma)$
be the number of classes $\alpha_i$ which are equal to $W$ and $F$ respectively,
and define a modified degree grading $\underline{\deg}$ by:
\[ \underline{\deg}(\gamma) = \deg(\gamma) + w(\gamma) - f(\gamma). \]

For a series $f = \sum_{d,r} c(d,r) q^d p^r$ define the formal Hecke operator by
\[
T_{m,\ell} f = \sum_{d,r} \left( \sum_{k|(m,d,r)} k^{\ell-1} c\left( \frac{md}{k^2}, \frac{r}{k} \right) \right) q^d p^r.
\]

\begin{lemma}
Conjecture~\ref{conjecture:MC} holds if and only if for all $m>0$, all $g, N, \alpha$ and all $(\deg, \underline{\deg})$-bihomogeneous classes $\gamma_i$ we have:
\begin{equation} F_{g,m}(\alpha; \gamma_1, \ldots, \gamma_N) =
m^{\sum_i \deg(\gamma_i) - \underline{\deg}(\gamma_i)} \cdot
T_{m,\ell} F_{g,1}(\alpha; \gamma_1, \ldots, \gamma_N)
\label{ellmc} \end{equation}
where $\ell = 2n(g-1) + \sum_i \underline{\deg}(\gamma_i)$.
\end{lemma}

\begin{proof}
Given the class $\beta = mW + dF + rA$ and a divisor $k | \beta$,
consider the real isometry $\varphi_k : H^2(X,\BR) \to H^2(X,\BR)$ defined by
\begin{align*}
 W & \mapsto \frac{k}{m} W \\
F & \mapsto \frac{m}{k} F \\ 
\gamma & \mapsto \gamma \text{ for all } \gamma \perp W,F.
\end{align*}
We extend this map to the full cohomology by
$\varphi_k = \rho(\phi_k,0): H^{\ast}(X,\BR) \to H^{\ast}(X,\BR)$.
We then have
\[ \varphi_k\left( \frac{\beta}{k} \right) = W + \frac{dm}{k^2} F + \frac{r}{k} A. \]
The Nakajima operators are equivariant with respect
to the action of $\varphi_k$ and the isometry of $H^{\ast}(S,\BR)$ given by
\[ \widetilde{\varphi}_k = \varphi_k|_{H^2(S,\BR)} \oplus \id_{H^0(S,\BR) \oplus H^4(S,\BR)}, \]
see Property 3 of Section~\ref{subsec:zariski closure}.

Let $\gamma_i \in H^{\ast}(X,\BQ)$ be elements in the Nakajima basis with respect to $\CB$.
If Conjecture~\ref{conjecture:MC} holds, then its application with respect to $\varphi_k$ yields:
\begin{multline} \label{abc}
F_{g,m}(\alpha; \gamma_1, \ldots, \gamma_N) =
\sum_{d,r} q^d p^r 
\sum_{k|(m,d,r)} (-1)^{r/k}
k^{3g-3+N-\deg(\alpha)} \\
\times \left( \frac{m}{k} \right)^{\sum_i f(\gamma_i) -w(\gamma_i)} \left\langle \alpha ; \gamma_1, \ldots, \gamma_N \right\rangle^{\Hilb_n}_{g, W + \frac{md}{k^2} F + \frac{r}{k} A}.
\end{multline}
Using the dimension constraint, our modified degree function and the formal Hecke operators this becomes
Define the weight of a cohomology class in the Nakajima basis with respect to $\CB$ by
\[ \wt\left( \prod_i \Fq_{a_i}(\alpha_i) 1 \right) = \sum_i \wt(\alpha_i) \]
where
\[ \wt(\alpha) =
 \begin{cases}
  1 & \text{ if } \alpha \in \{ \pt, W \} \\
  -1 & \text{ if } \alpha \in \{ 1, F \} \\
  0 & \text{ if } \alpha \in \{ 1,F,W,\pt \}^{\perp}.
 \end{cases}
\]
For a homogeneous $\gamma \in H^{\ast}(\Hilb_n(S))$ we then set $\underline{\deg}(\gamma) = \wt(\gamma) + n$.

Since the Nakajima operator $\Fq_i(\alpha)$ has degree $i-1 + \deg(\alpha)$ we have
for $\gamma = \prod_{i=1}^{\ell} \Fq_i(\alpha_i) 1 \in H^{\ast}(\Hilb_n S )$ that
\[ \deg( \gamma ) = n - \ell + \sum_{i} \deg(\alpha_i). \]
Hence we can rewrite:
\begin{align*}
\sum_{i=1}^{n} \deg(\gamma_i) + w - f & = n N - \sum_i \ell(\gamma_i) + \sum_{i,j} \deg(\alpha_{ij}) + w - f \\
& = n M + \sum_i \wt(\gamma_i) \\
& = \sum_i \underline{\deg}(\gamma_i)
\end{align*}
\begin{multline} \label{ml1}
F_{g,m}(\alpha; \gamma_1, \ldots, \gamma_N) = m^{\sum_i \deg(\gamma_i) - \underline{\deg}(\gamma_i)} \\
\times \sum_{d,r} q^d p^r
\sum_{k|(m,d,r)} 
(-1)^{r/k}
k^{2n(g-1)-1 + \sum_i \underline{\deg}(\gamma_i)}
\left\langle \alpha ; \gamma_1, \ldots, \gamma_N \right\rangle^{\Hilb_n}_{g, W + \frac{md}{k^2} F + \frac{r}{k} A}
\end{multline}
which is \eqref{ellmc}.

Conversely, equality \eqref{ellmc} implies Conjecture~\ref{conjecture:MC} since
any pair $(X,\beta)$ is deformation equivalent to some $(\Hilb_n(S), \widetilde{\beta} = mW+dF+kA)$ with $m>0$,
and the right hand side of Conjecture~\ref{conjecture:MC} is independent of choices.
\end{proof}

We will reinterpret the lemma in terms of Jacobi forms in \cite{K3nHAE}.
See also \cite{BB} for a parallel discussion in the case of K3 surfaces.

\section{Noether-Lefschetz theory} \label{sec:NL theory}

\subsection{Lattice polarized holomorphic-symplectic varieties}
Let $V$ be the abstract lattice, 
and let $L \subset V$ be a primitive non-degenerate sublattice\footnote{i.e. the quotient is torsion free.}.

An $L$-polarization of a holomorphic-symplectic variety $X$ is a primitive embedding
\[ j : L \hookrightarrow \Pic(X) \]
such that 
\begin{itemize}
\item there exists an isometry $\varphi : V \xrightarrow{\cong} H^2(X,\BZ)$ with $\varphi|_{L} = j$, and
\item the image $j(L)$ contains an ample class.
\end{itemize}
We  call the isometry $\varphi$ as above a marking of $(X,j)$.
If the image $j(L)$ only contains a big and nef line bundle,
we say that $X$ is $L$-quasipolarized.
Let $\CM_{L}$ be the moduli space of $L$-quasipolarized holomorphic-symplectic varieties of a given fixed deformation type.

The period domain associated to $L \subset V$ is
\[ \CD_{L} = \{ x \in \p(L^{\perp} \otimes \BC) \, | \, \langle x, x \rangle = 0, \langle x, \bar{x} \rangle > 0 \} \]
and has two connected components. Let $\CD_{L}^{+}$ be one of these components.
Consider the subgroup
\[ \Mon(V) \subset O(V) \]
which, for some choice of marking $\varphi : V \to H^2(X,\BZ)$ for some $(X,j)$ defining a point in $\CM_L$,
can be identified with the monodromy group $\Mon^2(X)$ of $X$.\footnote{If the monodromy group $\Mon(V)$ is normal in $O(V)$, then it
does not depend on the choice of $(X,j)$; this is known for all known examples of holomorphic-symplectic varieties.}
Let $\Mon(V)_L$ be the subgroup of $\Mon(V)$ that acts trivially on $L$.
Then the global Torelli theorem says that the period mapping
\[ \mathrm{Per} : \CM_{L} \to \CD_L^{+} / \Mon(V)_L \]
is surjective, restricts to an open embedding on the open locus of $L$-polarized holomorphic symplectic varieties,
and any fiber consists of birational holomorphic symplectic varieties, see \cite{MarkmanSurvey} for a survey and references.

\subsection{$1$-parameter families}
Let $L_i, i = 1, \ldots, \ell$ be an integral basis of $L$.

Given a compact complex manifold $\CX$ of dimension $2n+1$, line bundles
\[ \CL_1, \ldots, \CL_{\ell} \in \Pic(\CX) \]
and a morphism $\pi : \CX \to C$ to a smooth proper curve,
following \cite[0.2.2]{KMPS}, we call the tuple $(\CX, \CL_1, \ldots, \CL_{\ell}, \pi)$ 
a \emph{$1$-parameter family of $L$-quasi\-polarized holomorphic-symplectic varieties} if the following holds:
\begin{itemize}
\item[(i)] For every $t \in C$, the fiber $(X_t, L_i \mapsto \CL_i|_{X_t})$ is a $L$-quasipolarized holomorphic-symplectic variety.
\item[(ii)] There exists an vector $h \in L$ which yields a quasi-polarization on all fibers of $\pi$ simultaneously.
\end{itemize}
Any $1$-parameter family as above defines a morphism $\iota_{\pi} : C \to \CM_{L}$ into the moduli space of $L$-quasipolarized holomorphic-symplectic varieties
(of the deformation type specified by a fiber).

\subsection{Noether-Lefschetz cycles}
Given primitive sublattices $L \subset \widetilde{L} \subset V$,
consider the open substack
\[ \CM_{\widetilde{L}}' \subset \CM_{\widetilde{L}} \]
parametrizing pairs $(X, j : \widetilde{L} \hookrightarrow \Pic(X))$ such that
$j(L)$ contains a quasi-polarization.
There exists a natural proper morphism $\iota :  \CM_{\widetilde{L}}' \to \CM_{L}$ defined by restricting $j$ to $L$.
The Noether-Lefschetz cycle associated to $\widetilde{L}$ is the class of the reduced image of this map:
\[ \NL_{\widetilde{L}} = [ \iota( \CM_{\widetilde{L}}' ) ] \in A^c(\CM_L). \]
The codimension $c$ of the cycle is given by $\rank(\widetilde{L}) - \rank(L)$.
For $c=1$ we call $\NL_{\widetilde{L}}$ a Noether-Lefschetz divisor of the first type.

\subsection{Heegner divisors}
We review the construction of Heegner divisors. Their relation to Noether-Lefschetz divisors
will yield modularity results for intersection numbers with Noether-Lefschetz divisors. 

Consider the lattice
\[ M = L^{\perp} \subset V \]
and the subgroup
\[ \Gamma_M = \{ g \in O^{+}(M) | g \text{ acts trivially on } M^{\vee}/M \} \]
where $O^{+}(M)$ stands for those automorphisms which preserve the orientation, or equivalently, the component $\CD_L^{+}$.
We consider the quotient
\[ \CD_L^{+} / \Gamma_M. \]

For every $n \in \BQ_{<0}$ and $\gamma \in M^{\vee}/M$ the associated Heegner divisor is:
\[ y_{n,\gamma} = \left( \sum_{\substack{ v \in M^{\vee}  \\ \frac{1}{2} v \cdot v = n,\, [v] = \gamma } } v^{\perp} \right) / \Gamma_M. \]
For $n=0$ we define $y_{n,\gamma}$ by the
descent $\CK$ of the tautological line bundle $\CO(-1)$ on $\CD_L$ equipped with the natural $\Gamma_M$-action.
Concretely, we set
\[ y_{0,\gamma} = \begin{cases} c_1(\CK^{\ast}) & \text{ if } \gamma = 0 \\ 0 & \text{ otherwise.} \end{cases} \]
In case $n>0$ we set $y_{n,\gamma} = 0$ in all cases.

Define the formal power series of Heegner divisors
\[ \Phi(q) = \sum_{n \in \BQ_{\leq 0}} \sum_{\gamma \in M^{\vee}/M} y_{n,\gamma} q^{-n} e_{\gamma} \]
which is an element of $\Pic(\CD_L^{+} / \Gamma_M)[[q^{1/N}]] \otimes \BC[M^{\vee}/M]$,
where $e_{\gamma}$ are the elements of the group ring $\BC[M^{\vee}/M]$ indexed by $\gamma$ and
$N$ is the smallest integer for which $M^{\vee}(N)$ is an even lattice.

We recall the modularity result of Borcherds in the formulation of \cite{MP}:

\begin{thm}(\cite{Bor2, McGraw}) The generating series $\Phi(q)$ is the Fourier-expansion of a modular form
of weight $\rank(M)/2$ for the dual of the Weil representation $\rho_M^{\vee}$ of the metaplectic group $\mathrm{Mp}_2(\BZ)$:
\[ \Phi(q) \in \Pic(\CD_L^{+} / \Gamma_M) \otimes \Mod( \mathrm{Mp}_2(\BZ) , \rank(M)/2, \rho_M^{\vee} ). \]
\end{thm} 

The modular forms for the dual of the Weil representations can be computed easily
by a Sage program of Brandon Williams \cite{Williams}.

\subsection{Noether-Lefschetz divisors of the second type}
The precise relationship between Noether-Lefschetz and Heegner divisors for arbitrary holomorphic-symplectic varieties
is somewhat painful to state. For once the monodromy group $\Mon^2(X)$ is not known in general, and even if it is known
it usually does not contain $\Gamma_M$ or is contained in $\Gamma_M$.
To simplify the situation we from now on restrict to the case of $K3[n]$-type for $n \geq 2$.
Hence we let
\[ V = E_8(-1)^{\oplus 2} \oplus U^3 \oplus (2-2n) \]
and we fix an identification $V^{\vee}/V = \BZ/(2n-2)\BZ$.

We define the Noether-Lefschetz divisors of second type:
\[ \NL_{s,d, \pm r} \in A^1(\CM_L) \]
where $d = (d_1, \ldots , d_{\ell}) \in \BZ^{\ell}$, $s \in \BQ$ and $r \in \BZ/(2n-2)\BZ$ are given.
Consider the intersection matrix of the basis $L_i$:
\[ 
\mathsf{a} 
= \big( a_{ij} \big)_{i,j=1}^{\ell}
= \big( L_i \cdot L_j \big)_{i,j=1}^{\ell}. 
\]
We set
\[ \Delta(s, d) =
\det \begin{pmatrix}
a & d^t \\
d & s
\end{pmatrix}
= 
\det
\begin{pmatrix}
a_{11} & \cdots & a_{1 \ell} & d_1 \\
\vdots & & \vdots & \vdots \\
a_{\ell 1} & \cdots & a_{\ell \ell} & d_{\ell} \\
d_{1} & \cdots & d_{\ell} & s
\end{pmatrix}
\]

\vspace{5pt}
\noindent 
\textbf{Case: $\Delta(s, d) \neq 0$.} We define
\[ \NL_{s,d, \pm r} = \sum_{L\subset \tilde{L} \subset V} \mu(s,d, r | L\subset \tilde{L} \subset V ) \cdot \NL_{\tilde{L}} \]
where the sum runs over all isomorphism classes of primitive embeddings $L \subset \tilde{L} \subset V$
with $\rank(\tilde{L}) = \ell+1$. The multiplicity\footnote{For this construction
it would be more natural to work with pairs of a holomorphic-symplectic varieties and a primitive embedding
$j : L \to N_1(X)$ into the group of effective $1$-cycles $N_1(X) \subset H_2(X,\BZ)$.
If we then consider a rank $1$ overlattice $L \subset \widetilde{L} \subset N_1(X)$ we define the multiplicity $\mu$ as the number of $\beta \in \widetilde{L}$
such that $\beta \cdot L_i = d_i$, $\beta \cdot \beta = s$, and $[\beta] = \pm r$.
The condition above is more cumbersome but equivalent to this definition.}
$\mu(s,d, r | L\subset \tilde{L} \subset V )$
is the number of elements $\beta \in V^{\vee}$ which are contained in $\widetilde{L} \otimes \BQ$
and satisfy:
\[ \beta \cdot L_i = d_i, \quad \beta \cdot \beta = s, \quad \pm [\beta] = \pm r \text{ in } \BZ/(2n-2). \]
Here we have used the canonical embeddings $V^{\vee} \subset V \otimes \BQ$ and $\widetilde{L} \otimes \BQ \subset V \otimes \BQ$.

\vspace{5pt}
\noindent 
\textbf{Case: $\Delta(s, d) = 0$.} 
In this case any curve class with these invariants has to lie in $L \otimes \BQ$
and is uniquely determined by the degree $d$.
Hence we let $\beta \in L \otimes \BQ$
be the unique class so that $\beta \cdot L_i = d_i$ for all $i$.\footnote{The class is given by $\sum_{i,j=1}^{\ell} d_i (a^{-1})_{ij} L_j$.}
If $\beta$ lies in $V^{\vee}$ and has residue $[\beta] = \pm r$ we define
\[ \NL_{s,d, \pm r} = c_1(\CK^{\vee}), \]
and we define $\NL_{s,d, \pm r} = 0$ otherwise.

\begin{rmk}
Often the residue set $\pm [\beta]$ of a class $\beta \in H_2(X,\BZ)$ is determined by the degrees
$d_i = \beta \cdot L_i$.
For example, if $L$ contains a class $\ell$ such that
\[ \langle \ell, H^2(X,\BZ) \rangle = (2n-2) \BZ, \quad \langle \ell, H_2(X,\BZ) \rangle = \BZ \]
we may define a natural isomorphism by cupping with $\ell$:
\[ H_2(X,\BZ) / H^2(X,\BZ) \xrightarrow{\cong} \BZ / (2n-2)\BZ, \gamma \mapsto \gamma \cdot \ell. \]
(Not every polarization is of that form, for example the case of double covers of EPW sextics.)
In other cases the residue set is determined by the norm $\beta \cdot \beta$, for example in $K3^{[2]}$-type.
When the residue is determined by $s$ and $d$ we will drop it from the notation of Noether-Lefschetz divisors. \qed
\end{rmk}

%

\subsection{Heegner and Noether-Lefschetz divisors}
By the result of Markman, $\Mon(V) \subset O(V)$ is the subgroup
of orientation preserving isometries which act by $\pm \id$ on the discriminant.
Hence we have the inclusion 
$\Gamma_M \subset \Mon(V)_L$ of index $1$ or $2$. This yields the diagram:
\[
\begin{tikzcd} 
& \CD_L^{+} / \Gamma_M \ar{d}{\pi} \\ \CM_{L} \ar{r}{\mathrm{Per}} & \CD_L^{+} / \Mon(V)_L.
\end{tikzcd} \]
where $\pi$ is either an isomorphism or of degree $2$.

Let $C \subset \CM_{L}$ be a complete curve, and define the modular form
\[ \Phi_C(q) = \langle \Phi(q), \pi^{\ast} [\mathrm{Per}(C)] \rangle. \]
We write $\Phi_C[n,\gamma]$ for the coefficient of $q^n e^{\gamma}$ in the Fourier-expansion of $\Phi_C$.

We will need also:
\[ \widetilde{\Delta}(s, d) := -\frac{1}{2} \cdot \frac{1}{\det(a)} \det \begin{pmatrix}
a & d^t \\
d & s
\end{pmatrix}. \]
The following gives the main connection between the Noether-Lefschetz divisors
of the second type and the Heegner divisors.

\begin{prop}\label{prop:NL}
There exists a canonically defined class $\gamma(s,d,r) \in M^{\vee}/M$ 
(abbreviated also by $\gamma(r)$) such that
we have the following:
\begin{enumerate}
\item[(a)] If $\pi$ is an isomorphism,
\[ C \cdot \NL_{s,d, \pm r} = \Phi_C\left[ \widetilde{\Delta}(s,d) , \gamma(r)\right]. \]
\item[(b)] If $\pi$ is of degree $2$,
\[ C \cdot \NL_{s,d, \pm r}
=
\begin{cases}
\frac{1}{2}  \Phi_C\left[ \widetilde{\Delta}(s,d), \gamma(r) \right] & \text{ if } r=-r \\[10pt]
\frac{1}{2} \left( \Phi_C\left[ \widetilde{\Delta}(s,d), \gamma(r) \right] +  \Phi_C\left[ \widetilde{\Delta}(s,d), \gamma(-r) \right] \right) & \text{ otherwise }
\end{cases}
\]
\end{enumerate}
\end{prop}
\vspace{7pt}

In $K3^{[2]}$-type, we have $V^{\vee}/V = \BZ_2$
so that $\pi$ is an isomorphism. Moreover, the residue $r$ of any $\beta \in V^{\vee}$ is determined by its norm $s= \beta \cdot \beta$.
Hence omitting $r$ from the notation we find:
\begin{cor} \label{cor:NL_for_K3-2-type}
In $K3^{[2]}$ type, there exists a canonical class $\gamma = \gamma(d,s)$ with
\[ C \cdot \NL_{s,d} = \Phi_C[ \widetilde{\Delta}(s,d), \gamma ]. \]
\end{cor}

\vspace{5pt}
\begin{rmk} \label{remark:prop:NL}
In fact, in $K3^{[2]}$ type, the proof below will imply the equality of divisors
\[ \NL_{s,d} = \Phi[ \widetilde{\Delta}(s,d), \gamma ] = y_{-\tilde{\Delta}(s,d), \gamma} \]\
on $\CM_L$, where we have omitted the pullback by the period map $\mathrm{Per}$ on the right hand side.
\end{rmk}

For the proof of Proposition~\ref{prop:NL} we will repeatedly use the following basic linear algebra fact whose proof we skip.
\begin{lemma} \label{lemma:linear algebra}
Consider a $\BR$-vector space $\Lambda$ with inner product $\langle -, - \rangle$ and a orthogonal decomposition $L \oplus M = \Lambda$.
Let $L_i$ be a basis of $L$ with intersection matrix $a_{ij} = L_i \cdot L_j$.
For $\beta \in \Lambda$ with $d_i = \beta \cdot L_i$, let $v = \beta - \sum_{i,j} d_i a^{ij} L_j$ be the projection of $\beta$ onto $M$,
where $a^{ij}$ are the entries of $a^{-1}$.
Then we have
\[ \langle v,v \rangle = \frac{1}{\det(a)} \det \begin{pmatrix} a & d^t \\ d & \langle \beta, \beta \rangle \end{pmatrix}. \]
where $d = (d_1, \ldots, d_{\ell})$.
\end{lemma} 

The main step in the proof of the proposition is given by the following lemma:
For fixed $d = (d_1, \ldots, d_{\ell}), s$ and $r \in \BZ_{2n-2}$ consider the divisor on $\CD_L / \Gamma_M$ given by
\[
\NL_{s, d,r} = 
\left( \sum_{\beta} \beta^{\perp} \right)/\Gamma_M
\]
where the sum is over all classes $\beta \in V^{\vee}$ such that
\begin{equation} \beta \cdot \beta = s, \quad \beta \cdot L_i = d_i, i=1,\ldots, \ell, \quad\text{and} \quad [\beta] = r \in V^{\vee}/V. \label{beta condition} \end{equation}
Moreover, $\beta^{\perp}$ stands for the hyperplane in $\p(V)$ orthogonal to $\beta$ intersected with the period domain
$D_{L}$.

\begin{lemma} \label{lemma:comparision}
There exists a canonically defined class $\gamma = \gamma(s,d, r) \in M^{\vee}/M$ such that
\[ \NL_{s, d,r} = y_{n,\gamma} \ \in A^1(\CD_L/\Gamma_M) \]
where $n = \frac{1}{2} \frac{1}{\det(a)} \det \binom{a\ d}{d\ s}$.
\end{lemma}
\begin{proof}
In view of the definition of both sides of the claimed equation
it is enough to establish a bijection between
\begin{itemize}
\item[(a)] the set of classes $\beta \in V^{\vee}$ satisfying \eqref{beta condition}, and
\item[(b)] the set of classes $v \in M^{\vee}$ satisfying $v^2 =  \frac{1}{\det(a)} \det \binom{a\ d}{d\ s}$ and $[v]=\gamma$ for an appropriately defined $\gamma$.
\end{itemize}
Consider a primitive embedding $V \subset \Lambda$ into the Mukai lattice
\[ \Lambda = E_8(-1)^{\oplus 2} \oplus U^{\oplus 4}. \]
Let $e,f$ be a symplectic basis of one summand of $U$.
We choose the embedding such that $V^{\perp} = \BZ L_0$ where $L_0 = e+(n-1)f$. 
Since $\Lambda$ is unimodular, there exists a canonical isomorphism
\begin{equation} V^{\vee}/V \cong (\BZ L_0)^{\vee} / \BZ L_0 \label{asdasd}\end{equation}
and we may assume that under this isomorphism the class $L_0/(2n-2)$ mod $\BZ L_0$ corresponds to $1 \in \BZ/(2n-2) \BZ$.

\vspace{5pt} 
\noindent \emph{Step 1.} Let $d_0 \in \BZ$ be any integer such that $d_0 \equiv r$ modulo $2n-2$, and let
$\widetilde{s} \in 2 \BZ$ such that
\[ s = \frac{1}{2n-2} \det \begin{pmatrix} 2n-2 & d_0 \\ d_0 & \widetilde{s} \end{pmatrix} \ \Longleftrightarrow \ \widetilde{s} = s + \frac{d_0^2}{2n-2}. \]
(We may assume such $\widetilde{s}$ exists: Otherwise the set in (a) is empty, and by the argument below also the set in (b)).
Then we claim that there exists a bijection between the set in (a) and
\begin{itemize}
\item[(c)] the set of classes $\widetilde{\beta} \in \Lambda$ such that $\widetilde{\beta} \cdot L_i = d_i$ for all $i=0,\ldots, \ell$
and $\widetilde{\beta} \cdot \widetilde{\beta} = \widetilde{s}$.
\end{itemize}
\begin{proof}[Proof of Step 1]
Given $\widetilde{\beta}$ satisfying the conditions in (c) then
\[ \beta = \widetilde{\beta} - \frac{d_0}{2n-2} L_0 \]
lies in $V^{\vee}$. Moreover, $[\beta]$ is the class in $V^{\vee}/V$ corresponding to $d_0/(2n-2)$ in $(\BZ L_0)^{\vee} / \BZ L_0$ under \eqref{asdasd},
hence $[\beta] = r$. Also $\beta \cdot L_i=d_i$ for $i=1,\ldots, \ell$.
The equality $\beta \cdot \beta = s$ is by definition of $\widetilde{s}$ and Lemma~\ref{lemma:linear algebra}. 

Conversely, let $\delta = -e + (n-1) f$ and observe that $L_0 \cdot \delta = 0$ and $L_0 / (2n-2) + \delta/(2n-2) = f$.
Hence, if $\beta$ satisfies (a) then $\beta$ is an element of $d_0 \cdot \frac{\delta}{2n-2} + V$ and hence
\[ \beta \mapsto \beta + \frac{d_0}{2n-2} L_0 \in \Lambda \]
defines the required inverse.
\end{proof}

We consider now the embedding $M \subset \Lambda$ and the orthogonal complement
\[ \widehat{L} = M^{\perp}. \]
Since $\Lambda$ is unimodular, we have an isomorphism
\[ \widehat{L}^{\vee} / \widehat{L} \cong M^{\vee} / M. \]
We specify the class $\gamma$ via this isomorphism. Concretely, we set
\[ \gamma := \left[ \sum_{i,j = 0}^{\ell} d_i a^{ij} L_j \right] \in  \widehat{L}^{\vee} / \widehat{L} \]
where we let $a^{ij}$ denote the entries of the inverse of the extended intersection matrix $\hat{a} = ( L_i \cdot L_j )_{i,j=0, \ldots , \ell}$.

If we replace $d_0$ by $d_0 + (2n-2)$, then since $a^{00} = 1/(2n-2)$ and $a^{0j} = 0$ for $j \neq 0$,
the expression $\sum d_i a^{ij} L_j$ gets replaced by the same expression plus $L_0$.
Hence the class $\gamma$ only depends on $s,(d_1, \ldots, d_{\ell}),r$.

\vspace{5pt} 
\noindent \emph{Step 2.} There exists a bijection between the classes in (c) and (b).
\begin{proof}[Proof of Step 2]
We have the bijection
\[ \widetilde{\beta} \mapsto \widetilde{\beta} - \sum_{i,j=0}^{r} d_i a^{ij} L_j \in M^{\vee}. \]
\end{proof}
Combining Step 1 and 2 finished the proof of the lemma.
\end{proof}

\begin{proof}[Proof of Proposition~\ref{prop:NL}]
By definition we have:
\[ \NL_{s,d_1, \ldots, d_r, \pm r} = \mathrm{Per}^{\ast}\left[  \left( \sum_{\beta} \beta^{\perp} \right) / \Mon(V)_L \right] \]
where the sum is over all $\beta \in V^{\vee}$ such that
\[ \beta \cdot \beta = s, \quad \beta \cdot L_i = d_i, \quad [\beta ] = \pm r. \]
Hence if $\pi$ is an isomorphism, the result follows from this by Lemma~\ref{lemma:comparision}.

Hence assume now $\pi$ is of degree $2$, and let $\widetilde{C} = \mathrm{Per}(C)$.

If $-r = r$, then the morphism $\NL_{s,d, r} \to \NL_{s,d, \pm r}$ given by restriction of $\pi$ is of degree $2$.
Therefore
\[
C \cdot \NL_{s,d, \pm r} = \frac{1}{2} \widetilde{C} \cdot \pi_{\ast} \NL_{s,d, r} =
\frac{1}{2} \pi^{\ast}[\widetilde{C}] \cdot   \NL_{s,d, r}
\]
which then implies the claim by Lemma~\ref{lemma:comparision}.
If $r \neq -r$, then we have $\pi_{\ast}  \NL_{s,d, r} =  \NL_{s,d, \pm r}$
from which the result follows.
\end{proof}

\subsection{Noether-Lefschetz numbers}
Let $(\CX, \CL_1, \ldots, \CL_{\ell}, \pi)$ be a $1$-parameter family of $L$-quasipolarized holomorphic-symplectic varieties of $K3^{[n]}$-type.
We have the associated classifying morphism
\[ \iota_{\pi} : C \to \CM_{L}. \]
We define the Noether-Lefschetz numbers of the family by
\[ \NL^{\pi}_{s,d, \pm r} = \int_C \iota_{\pi}^{\ast}\NL^{\pi}_{s,d, \pm r}.  \]
Intuitively, the Noether-Lefschetz numbers are the number of fibers of $\pi$
for which there exists a Hodge class $\beta$ with prescribed norm $\beta \cdot \beta = s$, degree $\beta \cdot L_i = d_i$ and residue $[\beta ] = \pm r$.

In $K3^{[2]}$-type we will also often write $\Phi^{\pi}(q) = \iota_{\pi}^{\ast} \Phi(q)$.

The families of holomorphic-symplectic varieties we will encounter in geometric constructions
often come with mildly singular fibers.
The definition of Noether-Lefschetz numbers can be extended to these families as follows.
Let $\pi: \CX \to C$ be a projective flat morphism to a smooth curve and let $\CL_{1}, \ldots, \CL_\ell \in \Pic(\CX)$.
We assume that over a non-empty open subset of $C$ this defines a $1$-parameter family of $L$-quasipolarized holomorphic-symplectic varieties of $K3^{[n]}$ type.
We also assume that around every singular point the monodromy is finite.
Then there exists a cover
\[ f:\widetilde{C} \to C \]
such that the pullback family $f^{\ast} \CX \to \widetilde{C}$ is bimeromorphic to a
$1$-parameter family of $L$-quasipolarized holomorphic-symplectic varieties of $K3^{[n]}$-type,
\[ \widetilde{\pi} : \tilde{\CX} \to \widetilde{C}. \]
See for example \cite{KLSV}.
Concretely, around each basepoint of a singular fiber, after a cover that trivializes the monodromy,
the rational map $C \dashrightarrow \CM_L$ can be extended.
(In the examples we will consider, we can construct the cover $\tilde{C} \to C$ and the birational model $\tilde{\CX}$ explicitly).
We define the Noether-Lefschetz numbers of $\pi$ by:
\[ \NL^{\pi}_{s,d,\pm r} := \frac{1}{k} \NL^{\widetilde{\pi}}_{s,d,\pm} \]
where $k$ is the degree of the cover $\widetilde{C} \to C$.
Since the Noether-Lefschetz divisors are pulled back from the separated period domain,
the definition is independent of the choice of cover.

\subsection{Example: Prime discriminant in $K3^{[2]}$-type} \label{ex: prime discriminant}
Let
\[ V = E_8(-1)^{\oplus 2} \oplus U^{\oplus 3} \oplus \BZ \delta, \quad \delta^2 = -2 \]
be the $K3^{[2]}$-lattice and consider a primitive vector $H \in V$ satisfying
\begin{itemize}
\item $H \cdot H = 2p$ for a prime $p$ with $p \equiv 3$ mod $4$,
\item $\langle H, V \rangle = 2 \BZ$.
\end{itemize}
Equivalently, $H/2$ defines a primitive vector in $V^{\vee}$ and has norm $p/2$.
By Eichler's criterion \cite[Lemma 3.5]{GHS} there exists a unique $O(V)$ orbit of vectors $H$ satisfying these condition.
To be concrete we choose
\[ H = 2\left( e' + \frac{p+1}{4} f' \right) + \delta \]
where $e',f'$ is a basis of one of the summands $U$. In $V$ one then has
\[ M = H^{\perp} \cong E_8(-1)^{\oplus 2} \oplus U^{\oplus 2} \oplus \begin{pmatrix} -2 & -1 \\ -1 & -\frac{p+1}{2} \end{pmatrix}, \]
a lattice of discriminant group $\BZ/p \BZ$. 

We consider $H$-quasipolarized holomorphic-symplectic varieties $X$ of $K3^{[2]}$-type.
Examples are the Fano varieties of lines ($p=3$) or the Debarre-Voisin fourfolds ($p=11$), see below.
For these varieties the Borcherds modular forms and the relationship between between Noether-Lefschetz divisors of first and second type can be described very explicitly.

\subsubsection{The Borcherds modular forms} \label{subsubsec:Borcherds form}
Consider the series of Noether-Lefschetz numbers of second type
\[ \Phi^{\pi}(q) 
= \sum_{\gamma} \Phi^{\pi}_{\gamma}(q) e_{\gamma} \]
for a $1$-parameter family $\pi$ of holomorphic-symplectic varieties of this polarization type.
This is a modular form of weight $11$ for the Weil representation on $M^{\vee}/M$.
The space of such forms is easily computed through \cite{Williams} and the first values are 
given in the following table. 

\begin{table}[ht] 
\label{table:dimension}
{\renewcommand{\arraystretch}{1.5}\begin{tabular}{| c | c | c | c | c | c | c | c | c | c | c |}
\hline
$\!p$\! & $3$ & $7$ & $11$ & $19$ & $23$ \\
\hline
$\dim$ & $2$ & $4$ & $6$ & $9$ & $12$ \\
\hline
\end{tabular}}
\vspace{8pt}
\caption{The dimension of the space of modular forms of weight $11$ for the Weil representation associated to $M$.} \label{uniruled_table}
\end{table}

If we write $y_1, y_2$ for the standard basis of the lattice $\begin{pmatrix} -2 & -1 \\ -1 & -\frac{p+1}{2} \end{pmatrix}$,
then the discrimimant of $M$ is generated by
\[ y' = \frac{1}{p} (2 y_2-y_1) \]
which has norm $y' \cdot y' = -2/p$. 
Hence for any element $v$ of $M^{\vee}$, written as
\[ v = w + k y' \in M^{\vee}, \quad w \in M, k \in \BZ, \]
we have $-\frac{1}{2} p v \cdot v = k^2$ modulo $p$.
In particular, this determines $[v] \in \BZ/p\BZ$ up to multiplication by $\pm 1$.
Thus for any $v \in M^{\vee}$ we see that:
\begin{enumerate}
\item[(i)] $D := -\frac{p}{2} v \cdot v$ is a square modulo $p$, and 
\item[(ii)] $r=[v]$ is determined from $D$ via $r^2 \equiv D$ mod $p$,
up to multiplication by $\pm 1$.
\end{enumerate}
By the redundancy of Heegner divisors $y_{n,\gamma} = y_{n,-\gamma}$, the coefficient
$q^n e_{\gamma}$ of $\Phi(q)$ is thus determined by $n$ alone.
It is hence enough to consider
\begin{equation} \varphi^{\pi}(q) = \frac{1}{2} \Phi^{\pi}_0(q) + \frac{1}{2} \sum_{\gamma \in M^{\vee}/M} \Phi^{\pi}_{\gamma}(q). \label{phi series} \end{equation}

Let $\chi_{p}$ be the Dirichlet character given by the Legendre symbol $\left( \frac{ \cdot }{p} \right)$.
\begin{prop} \label{prop:phi modularity} The series
$\Phi^{\pi}_0(q)$ and $\sum_{\gamma \in M^{\vee}/M} \Phi^{\pi}_{\gamma}(q^p)$
are modular forms of weight $11$ and character $\chi_p$ for the congruence subgroup $\Gamma_0(p)$.
\end{prop}
\begin{proof}
The modularity of the first series is well-known \cite{Borcherds}.
The second is one direction of the Bruinier-Bundschuh isomorphism \cite{BruiBund}.
\end{proof}

The generators of the ring of modular forms for the character $\chi_p$ is easily computable (see e.g. \cite[Sec.12]{Borcherds})
which yields explicit formulas for $\varphi^\pi$.
One example for Fano varieties can be found in \cite{LZ}. We will consider the case of Debarre-Voisin fourfolds below.

Finally, the Noether-Lefschetz numbers of the family are given by:
\[ 
\NL^{\pi}_{s,d} = \varphi^{\pi}\left[ -\frac{1}{4p} \det \begin{pmatrix} 2p & d \\ d & s \end{pmatrix} \right].
\]

\subsubsection{Noether-Lefschetz divisors the first type} \label{subsec:NL divisors first type}
The relationship between Noether-Lefschetz divisors of the first and second type is
not so easy to state in general. However, here the situation simplifies.
For any $w \in H^{\perp} \subset V$ we consider the intersection of $w^{\perp}$ with the period domain $\CD_H$,
\[ \CD_{w^{\perp}} = \{ x \in \CD_{H} | \langle x, w \rangle = 0 \}. \]
The image of this divisor under the quotient map $\CD_{H} \to \CD_{H}/\Gamma_{M}$
defines an irreducible divisor that by a result of Debarre and Macr\`i \cite{DM}
only depends on the discriminant
\[ -2e := \mathrm{disc}( w^{\perp} \subset M ). \]
Moreover, $e$ is a square modulo $p$.
We write $\CC_{2e}$ for this divisor. 

The relationship between Noether-Lefschetz divisors of first and second type is given as follows:
\begin{prop} \label{prop:NL first and second}
Let $D \geq 1$ be a square modulo $p$, and let $\alpha \in \BZ/p \BZ$ such that $\alpha^2 \equiv D$ mod $p$.
The associated Heegner divisor $y_{-D/p,\alpha}$, denoted also by $\NL(D)$, is given by
\[ 
\NL(D)
= \sum_{\substack{a_0 \geq 0, k \in \{ 0, \ldots, \floor{\frac{p}{2}} \} \\
e = p a_0 + k^2 \geq 1}}
\left| \left\{ c \in \BZ\, \middle|\, c^2 = \frac{D}{e},\ kc \equiv \alpha \textup{ mod } p \right\} \right| \CC_{2e}
\]
\end{prop}

\vspace{5pt}
In particular, we have
\[
\NL(D) =
\begin{cases}
\CC_{2D} + \ldots & \text{ if } D \neq 0 \text{ mod } 11 \\
2 \CC_{2D} + \ldots & \text{ if } D = 0 \text{ mod } 11
\end{cases}
\]
where $\ldots$ stands for terms $\CC_{2e}$ with $e<D$.
This shows that the Noether-Lefschetz divisors of the first type are related to the Heegner divisors by an invertible upper triangular matrx.
If $D$ is square free, then $\NL(D)$ and $\CC_{2D}$ agree up to a constant.

\begin{proof}
For any positive $e = p a_0 + k^2$ with $k \in \{ 0, 1 \ldots, \floor{\frac{p}{2}} \}$ and $a \geq 1$
we choose a lattice $K_e \subset V$ containing $H$ and such that $\mathrm{disc}(K_e^{\perp}) = -2e$.
The lattice is unique up to an automorphism of $V$ that fixes $H$ \cite{DM}.
Fix $s \in \frac{1}{2} \BZ$ with $2s \equiv 3 (4)$ and $d \geq 1$ such that $D = -\frac{1}{4} \det\binom{2p\ d}{d\ s} = \frac{1}{4}( d^2 - 2ps )$.
Then by Proposition~\ref{prop:NL}, Remark~\ref{remark:prop:NL} and the definition we have
\[ \NL(D) = \NL_{s,d} = \sum_{e} \mu(K_e, s, d) \CC_{2e} \]
where the multiplicity is given by
\begin{equation}  \mu(K_e, s, d) = | \{ \beta \in K_e \otimes \BQ | \beta \in V^{\vee}, \beta \cdot \beta = s, \beta \cdot H = d \} |. \label{dadg3} \end{equation}

It remains to calculate the multiplicty.
We first embed $V$ into the Mukai lattice $\Lambda$ as the orthogonal of $e+f$ such that $\delta = -e+f$.
Here $e,f$ is a symplectic basis of a not previously used copy of $U$.
One finds that
\[ \widehat{L} = (M^{\perp} \subset \Lambda) \cong \begin{pmatrix} 2 & 1 \\ 1 & \frac{p+1}{2} \end{pmatrix} \]
which has the integral basis
\[ x_1 = e+f, \quad x_2 = e' + \frac{p+1}{4} f' + f. \]
Let us next choose
\[ K_e = \BZ H \oplus \BZ (k f' + e'' - a_0 f'') \]
where $e'', f''$ is a symplectic basis of a third copy of $U$.
The saturation of $K_e \oplus \BZ (e+f)$ inside $\Lambda$ is then given by
\[
\widetilde{K}_e \cong
\begin{pmatrix} 2 & 1 & 0 \\ 1 & \frac{p+1}{2} & k \\ 0 & k & -2 a_0 \end{pmatrix}
\]
where the lattice is generated by $x_1, x_2$ and $x_3 = k f' + e'' - a_0 f''$.

We follow the recipe of the proof of
Lemma~\ref{lemma:comparision}, that is we compare the multiplicity \eqref{dadg3}
with a simpler multiplicity for $\tilde{K}_e$.
If $s \in 2 \BZ$, then for $D$ to be an integer, we must have $d$ even. Then as in Lemma~\ref{lemma:comparision} one gets:
\[
\mu(K_e, s, d)
=
\left| \left\{ \beta \in \widetilde{K}_e \middle| \beta \cdot x_1 = 0,\ \beta \cdot x_2 = \frac{d}{2}, \ \beta \cdot \beta = s \right\} \right|
\]
If $s+\frac{1}{2} \in \BZ$, then $d$ is odd and
\[ 
\mu(K_e, s, d)
=
\left| \left\{ \beta \in \widetilde{K}_e \middle| \beta \cdot x_1 = 1,\ \beta \cdot x_2 = \frac{d+1}{2}, \ \beta \cdot \beta = s+\frac{1}{2} \right\} \right|
\]

The result follows from this by a direct calculation.
For exposition we evaluate the multiplicity in the first case.
Using that $\beta \cdot x_1 = 0$, any element $\beta \in \tilde{K}_e$ as on the right hand side is given by
\[ \beta = a (-x_1 + 2 x_2) + c x_e. \]
Let $\tilde{d} = d/2$. The condition $\beta \cdot x_2 = \tilde{d}$ yields
$a p + kc = \tilde{d}$
which can be solved if and only if $kc \equiv \tilde{d}$ mod $p$, in which case $a = (\tilde{d}-kc)/p$.
Inserting this expression into $\beta \cdot \beta$ yields
\[ c^2 e = \tilde{d}^2 - \frac{p}{2} s = -\frac{1}{4}(2ps-d^2) = D. \]
Finally, $D = \tilde{d}^2$ mod $p$, and hence if $\alpha^2 = D$ mod $p$, then $\alpha = \pm \tilde{d}$.
If $\alpha \equiv 0$ mod $p$, then the result follows. In the other case,
among $c \in \{ \pm \sqrt{D/e} \}$ there is
precisely one solution to $kc \equiv \tilde{d}$ mod $p$ if and only if there is precisely one solution to
$kc \equiv \alpha$ mod $p$.
\end{proof}

\subsection{Example: Cubic fourfolds} \label{sec:example cubics}
We consider Fano varieties of lines 
\[ X \subset \Gr(2,6) \]
of a cubic fourfold.
By \cite{BD} the Pl\"ucker polarization is of square $6$ and
pairs evenly with any class in $H^2(X,\BZ)$.
Hence their deformation type is governed by the discussion in Section~\ref{ex: prime discriminant} for $p=3$.
The Borcherds modular form for the generic pencil of Fano varieties is computed in \cite{LZ}.

Let $U \subset \p( H^0(\p^5, \CO(3)))$ be the open locus corresponding to cubic fourfolds with at worst ADE singularities.
There is a period mapping
\[ p : U \to \CM_{H} \]
to the corresponding moduli space. The pullback of the divisors $\CC_{2e}$ under this mapping
are the special cubic fourfolds of discriminant $d=2e$, see \cite{Hassett, LZ}.
(A cubic fourfold $Y \subset \p^5$ is special if it contains an algebraic surface $S$ such that the saturation of $[S]$ and $h^2$ is of discriminant $d$).

For the $1$-parameter family $\pi$ of Fano varieties of lines of a generic pencil of cubic fourfolds the
Noether-Lefschetz numbers of the second type $\NL_{s,d}^{\pi}$ and of first type
\begin{equation} \NL^{\pi}(D) = \deg \iota_{\pi}^{\ast} \NL(D) \label{nldiv} \end{equation}
are then related to the classical geometry of special cubic fourfolds.
For example,
\begin{gather*}
\NL^{\pi}_{-2,0} = \NL^{\pi}(D = 3) = 192 \\
\NL^{\pi}_{-2,4} = \NL^{\pi}(D = 7) = 917568
\end{gather*}
are the degrees of the (closure of the) divisors in $\p( H^0(\p^5, \CO(3)))$ parametrizing nodal and Pfaffian cubics respectively.
The locus of determinantal cubic fourfolds $p^{-1} \CC_{2}$ is of codimension $\geq 2$, see e.g. \cite[Rmk 3.23]{HuybrechtsNotes},
and hence 
\[ \NL^{\pi}_{-1/2,1} = \NL^{\pi}(D=1) = 0. \]
Thus one gets that
\[ \NL^{\pi}_{-5/2,1} = \NL^{\pi}(D = 4) = 3402 \]
which is the degree of the locus $p^{-1} \CC_{8}$ of cubics containing a plane.
The equalities of the Noether-Lefschetz numbers of first and second type above follow from Proposition~\ref{prop:NL first and second}:
in the first three cases since $D$ is square free, and in the last case we use that $\CC_{2}$ does not meet the curve defined by $\pi$.

\subsection{Example II: Debarre-Voisin fourfolds}
A Debarre-Voisin fourfold \cite{DV} is the holomorphic-symplectic variety 
\[ X \subset \Gr(6,10) \] given as the
vanishing locus of a section of $\Lambda^3 \CU^{\vee}$, where $\CU \subset \BC^{10} \otimes \CO$ is the universal subbundle on the Grassmannian.
These varieties are of $K3^{[2]}$-type
and the Pl\"ucker polarization is of degree $H^2 = 22$ and pairs evenly with any class in $H^2(X,\BZ)$.
Hence we are in the situation of
Section~\ref{ex: prime discriminant} for $p=11$.
The Noether-Lefschetz numbers for a generic pencil of these varieties will be computed below.

\subsection{Refined Noether-Lefschetz divisors} \label{subsec:refined NL}
We will need refined Noether-Lefschetz divisors
which also depend on the divisibility $m \geq 1$ of the curve class.
Refined Noether-Lefschetz numbers are then defined as usualy by intersection with Noether-Lefschetz divisors.
As before we assume that we are in $K3^{[n]}$-type.
Let $s \in \BQ$, $d = (d_1, \ldots, d_{\ell}) \in \BZ^{\ell}$ and $r \in \BZ_{2n-2}$ be fixed.

If $\Delta(s, d) \neq 0$ we set
\[ \NL_{m, s,d, \pm r} = \sum_{L\subset \tilde{L} \subset V} \mu(m, s,d, r | L\subset \tilde{L} \subset V ) \cdot \NL_{\tilde{L}} \]
where the refined multiplicity $\mu( \ldots )$ is the number of classes $\beta \in V^{\vee}$ which are contained in $\widetilde{L} \otimes \BQ$,
satisfy $\beta \cdot \beta = s$, $\beta \cdot L_i = d_i$ and such that the following new conditions hold:
\[ \div(\beta) = m, \quad \left[ \frac{\beta}{\div(\beta)} \right] = \pm r \in \BZ/(2n-2)\BZ. \]
Note that we treat the residue different from the non-refined case.

If $\Delta(s, d)  = 0$ we define
\begin{equation*} \label{553}
\NL_{m, s,d, \pm r} := \NL_{s, d, \pm m \cdot r} 
\end{equation*}
if $m$ is the gcd of $d_1, \ldots, d_r$ and the unique class $\beta \in L \otimes \BQ$
with $\beta \cdot L_i = d_i/m$ lies in $V^{\vee}$ and has residue $[\beta] = \pm r$.
Otherwise, we set $\NL_{m, s,d, \pm r} = 0$.

We then have
\begin{equation} \NL_{s,d, \pm r} = 
\sum_{m \geq 1} \sum_{\substack{\pm r'\\ \pm m \cdot r' = \pm r}}
\NL_{m, s,d, \pm r'}.
\label{refinedunrefined}
\end{equation}
and
\[ \NL_{m,s,d, \pm r} = \NL_{1, s/m^2, d/m, \pm r}. \]
By a simple induction argument as in \cite[Lemma 1]{KMPS},
these two equations show that the data of the unrefined Noether-Lefschetz numbers
are equivalent to the the data of the refined Noether-Lefschetz numbers/divisors.

\begin{rmk}
If the residue of a class is determined by $d$ and $s$, the inverse relation between refined and unrefined
is easy to state. We simply have
\[ \NL_{1,s,d} = \sum_{k|\mathrm{gcd}(d_1, \ldots, d_{\ell})} \mu(k) \cdot \NL_{s/k^2, d/k}, \]
parallel to the multiple cover rule we study in this paper.
\end{rmk}

\section{Gromov--Witten theory and Noether-Lefschetz theory} \label{sec:GWNL theory}
Let $V$ be the $K3^{[n]}$-lattice and let $L \subset V$ be a fixed primitive sublattice with integral basis $L_i$.
We consider a $1$-parameter family
\[ \pi : \CX \to C,\quad \CL_1 ,\dots, \CL_{\ell} \in \Pic(\CX) \]
of $L$-quasi\-polarized holomorphic-symplectic varieties of $K3^{[n]}$-type.

The goal of this section is to relate Gromov-Witten invariants of $\CX$ in fiber classes
to the Noether-Lefschetz numbers of the family
and the reduced Gromov-Witten invariants in $K3^{[n]}$-type.

\subsection{Gromov-Witten invariants of the family}
Let $\gamma_i \in H^{\ast}(\CX)$ be cohomology classes which can be written
in terms of polynomials $p_i$ in the Chern classes of $\CL_i$,
\[ \gamma_i = p_i(c_1(\CL_1), \ldots , c_1(\CL_{\ell}) ). \]
Let $\Mbar_{g,N}(\CX, d)$ for $d \in \BZ^{\ell}$ be the moduli space of $N$-marked genus $g$ stable maps $f : C \to \CX$
such that
\begin{itemize}
\item $f$ maps into the fibers of $\CX$, that is $\pi_{\ast} f_{\ast}[C] = 0$, and
\item $f$ is of degree $d_i$ against $L_i$,
\[ \int_{[C]} f^{\ast}(c_1(\CL_i)) = d_i. \]
\end{itemize}
We consider the invariants
\[
\blangle \alpha ; \gamma_1, \ldots, \gamma_N \brangle^{\CX}_{g,d}
=
\int_{[\Mbar_{g,n}(\CX, d)]} \tau^{\ast}(\alpha) \ev_{1}^{\ast}(\gamma_1) \cdots \ev_{N}^{\ast}(\gamma_N)
\]
where $\alpha \in H^{\ast}(\Mbar_{g,n})$ is tautological and $\tau$ is the forgetful map.

\subsection{Gromov-Witten invariants of the fiber}
Let $X$ be any holomorphic-symplectic variety
of $K3^{[n]}$-type and let $\beta \in H_2(X,\BZ)$ be an effective curve class.
Assume there exists an embedding
\[ L \otimes \BR \hookrightarrow H^2(X,\BR) \]
which is an isometry onto its image such that $\beta \cdot L_i = d_i$ for all $i$.
As usual we let $L_i \in H^2(X,\BR)$ denote the image of $L_i \in L$ under this map.
Let also
\[ \gamma_i = p_i( L_1, \ldots, L_{\ell} ). \]
By deformation invariance and the invariance property of Section~\ref{subsec:invariance}
the reduced Gromov--Witten invariant
$\blangle \alpha ; \gamma_1, \ldots, \gamma_N \brangle^X_{g,\beta}$ only depends on the degree $d = (d_1, \ldots, d_{\ell})$,
the polynomials $p_i$, $s = \beta \cdot \beta$, and the curve invariants $m = \div(\beta)$ and the residue set $\pm r = \pm [ \beta / \div(\beta) ]$.
We write
\[ \blangle \alpha ; \gamma_1, \ldots, \gamma_N \brangle^X_{g,\beta} = \blangle \alpha ; p_1, \ldots, p_N \brangle^X_{g,m,s,d,\pm r}. \]

\subsection{The relation}
Consider the refined Noether-Lefschetz numbers of $\pi$,
\[ \NL^{\pi}_{m, s,d, \pm r} := \int_{C} \iota_{\pi}^{\ast} \NL_{m, s,d, \pm r'} \]
where $\iota_{\pi}: C \to \CM_{L}$ is the morphism defined by the family.

\begin{prop} \label{prop:GWNLrelation}
Let $\gamma_i = p_i(\CL_1, \ldots, \CL_{\ell}) \in H^{\ast}(\CX)$. Then we have:
\[
\blangle \alpha ; \gamma_1, \ldots, \gamma_N \brangle^{\CX}_{g,d}
= 
\sum_{m,s, \pm r}
\NL^{\pi}_{m,s, d, \pm r} \cdot \blangle \alpha \cdot (-1)^g \lambda_g ; p_1, \ldots, p_N \brangle^X_{g,m,s,d,\pm r}
\]
\end{prop}

Here $\lambda_i$ are the $i$-th Chern classes of the Hodge bundle on the moduli space of stable curves.
The proposition can be extended to more general classes $\gamma_i$.
It is enough to assume that $\gamma_i$ is the product of some polynomial in the $\CL_i$
and a class which restricts to a monodromy invariant class on each fiber, for example a Chern class.

\begin{proof}
The proof follows by the identical argument as for the K3 surfaces, as discussed in \cite[Section 3.2]{MP}.
The above equality in the K3 case is \cite[Eqn. (17)]{MP}.
As in \cite{MP}, for each $\xi \in C$ we want to group together all curve classes in $H_2(\CX_{\xi},\BZ)$
of degree $d$ which have the same Gromov-Witten invariants. 
By Corollary~\ref{cor:curve deformations invariance}
we hence may group together classes of the same square, the same divisibility, and the same residue.
Thus we replace the set $B_{\xi}(m,h,d)$ of \cite[Sec.3.2]{MP} by
\[
B_{\xi}(m,s,d, \pm r) = \left\{ \beta \in H_2(\CX_{\xi},\BZ) \middle| 
\begin{array}{c} (\beta,\beta)=s,\ \mathrm\div(\beta) = m, \\ {[} \beta/\div(\beta)] = \pm r,\  \beta \cdot L_i = d_i \\ \beta \perp H^{2,0}(\CX_{\xi},\BC) \end{array}
\right\}
\]
The rest of the argument of \cite[Sec.3]{MP} goes through without change.
\end{proof}

\subsection{Reformulation}
We can rewrite Proposition~\ref{prop:GWNLrelation} in terms of invariants where we have formally subtracted multiple cover contrubtions.
For simplicity assume that for $\beta \in H_2(X,\BZ)$
the residue $r([\beta])$ is determined by the degrees $d_i = \beta \cdot L_i$.
Write $r(d)$ for the residue.
Proposition~\ref{prop:GWNLrelation} then says that
\[
\blangle \alpha ; \gamma_1, \ldots, \gamma_N \brangle^{\CX}_{g,d}
= 
\sum_{m,s}
\NL^{\pi}_{m,s, d} \cdot \blangle \alpha (-1)^g \lambda_g ; p_1, \ldots, p_N \brangle^X_{g,m,s,d}.
\]

Let us subtract formally the multiple cover contributions from the invariants of $\CX$,
\begin{multline*}
\blangle \alpha ; \gamma_1, \ldots, \gamma_N \brangle^{\CX, \text{mc}}_{g,d} \\
:=
\sum_{k|d} (-1)^{r(d) + r(d/k)} \mu(k) k^{2g-3+N-\deg(\alpha)} 
\blangle \alpha ; \gamma_1, \ldots, \gamma_N \brangle^{\CX}_{g,d/k}
\end{multline*}
as well as from the the reduced Gromov-Witten invariants,
\begin{multline*}
\blangle \alpha (-1)^g \lambda_g ; p_1, \ldots, p_N \brangle^{X,\text{mc}}_{g,m,s,d} \\
:= 
\sum_{k|m} (-1)^{r(d) + r(d/k)} k^{2g-3+N-\deg(\alpha)} \mu(k) \blangle \alpha (-1)^g \lambda_g ; p_1, \ldots, p_N \brangle^X_{g,m/k,s/k^2,d/k}.
\end{multline*}
The following is the result of a short calculation:

\begin{lemma} We have
\[
\blangle \alpha ; \gamma_1, \ldots, \gamma_N \brangle^{\CX, \text{mc}}_{g,d}
=
\sum_{m,s} \NL_{m,s,d} \cdot 
\blangle \alpha (-1)^g \lambda_g ; p_1, \ldots, p_N \brangle^{X,\text{mc}}_{g,m,s,d}
\]
\end{lemma} 

In particular, if the multiple cover conjecture (Conjecture~\ref{conjecture:MC})
holds, after subtracting the multiple cover contributions the invariants of $X$ do not depend on the divisibility $m$ and
so with \eqref{refinedunrefined} we obtain:
\begin{equation} \label{bps formulation:mc known}
\begin{aligned}
\blangle \alpha ; \gamma_1, \ldots, \gamma_N \brangle^{\CX, \text{mc}}_{g,d}
& =
\sum_{s} \blangle \alpha (-1)^g \lambda_g ; p_1, \ldots, p_N \brangle^{X,\text{mc}}_{g,1,s,d}
\sum_{m,s} \NL_{m,s,d}  \\
& =
\sum_{s} \NL_{s,d} \cdot \blangle \alpha (-1)^g \lambda_g ; p_1, \ldots, p_N \brangle^{X,\text{mc}}_{g,1,s,d}.
\end{aligned}
\end{equation}

\section{Mirror symmetry} \label{sec:mirror_symmetry}
\subsection{Overview}
In this section we review how to use mirror symmetry formulas
to compute the genus $0$ Gromov-Witten invariants
for the total space $\CX$ of generic pencils of Fano varieties of lines of cubic fourfolds
and of Debarre-Voisin varieties.

Mirror symmetry here means an application of the following results:
Givental's description of the $I$-function for complete intersections in toric varieties \cite{Givental},
the proof of the abelian/non-abelian correspondence by Webb
that relates the $I$-function of a GIT quotient with that of its abelian quotient \cite{Webb},
and the genus $0$ wallcrossing formula between quasi-maps and Gromov-Witten invariants for GIT quotients by Ciocan-Fontanine and Kim \cite{CFK}.

We first determine the small $I$-function
for the cases we are interested in, then we shortly recall how to relate the $I$ and $J$ functions.
We assume basic familiarity with the language of \cite{CFK, Webb} throughout.

\subsection{$I$-functions}
We work in the following setup: Let $V$ be a vector space over $\BC$,
and let $G$ be a connected reductive group acting faithfully on $V$ on the left.
We also fix a character of $G$ for which we assume that the semistable and stable locus, denoted by $V^s(G)$, agrees.
For simplicity we also assume that the $G$-action on the stable locus is free.
We consider the GIT quotient
\[ Y = V // G = V^s(G) / G. \]

Let $T \subset G$ be a maximal torus and consider also the abelian quotient $V^s(T)/T$. 
We have then the following diagram relating the abelian and non-abelian quotients: 
\[
\begin{tikzcd}
V^s(G)/T \ar{r}{j} \ar{d}{\xi} & V^s(T)/T \\
V^s(G)/G.
\end{tikzcd}
\]
The Weyl group $W$ of $G$ acts naturally on the cohomology of $V^s(G)/T$ and one has the isomorphism:
\[ \xi^{\ast} : H^{\ast}(V^s(G)/G , \BQ) \xrightarrow{\cong} H^{\ast}( V^s(G)/T, \BQ )^{W}. \]

Let $E$ be a $G$-representation and consider a smooth zero locus of the associated homogeneous bundle $E$ on $Y$,
\[ X \subset Y. \]
The small $I$-function of $X$ in $Y$ is a formal series
\[ I^{X} = I^{Y,E} = 1 + \sum_{\beta \neq 0} q^\beta I_{\beta}(z) \]
where $\beta \in H_2(Y,\BZ)$ runs over all curve classes, $q^{\beta}$ is a formal variable and $I_{\beta}(z)$ is a formal series in $z^{\pm 1}$ with coefficients in $H^{\ast}(Y,\BQ)$.
It can then be determined in the following steps:

\vspace{5pt}
\noindent 
\textbf{Abelian/Non-Abelian correspondence}(\cite{Webb}).
\[ \xi^{\ast} I_{\beta}^{Y,E} = j^{\ast} \sum_{\tilde{\beta} \mapsto \beta} \left( \prod_{\alpha} 
\frac{\prod_{k=-\infty}^{\tilde{\beta} \cdot c_1(L_\alpha)} (c_1(L_{\alpha}) + kz)}{\prod_{k=-\infty}^{0} (c_1(L_{\alpha}) + kz)} \right) I_{\tilde{\beta}}^{V//T,E} \]
where $\alpha$ runs over the roots of $G$ and $L_{\alpha}$ is the associated line bundle on the abelian quotient,
$\tilde{\beta} \in H_2(V//T,\BZ) = \Hom(\chi(T),\BZ)$ runs over the characters of $T$ that restrict to the given character $\beta \in H_2(Y,\BZ) = \Hom(\chi(G),\BZ)$ under the map induced by $\chi(G) \to \chi(T)$.
When it is clear from context, we will often omit the pullbacks $\xi^{\ast}$ and $j^{\ast}$ from the notation.

\vspace{5pt}
\noindent 
\textbf{Twisting}(\cite{Givental}). When restricting the $G$-representation $E$ to $T$, it decomposes into a direct sum of $1$-dimensional representations $M_i$.
We write $M_i$ also for the associated line bundles on $V//T$. Then
\[ I^{V//T,E}_{\beta} = \left( \prod_{i=1}^{\mathrm{rk}(E)} \prod_{k=1}^{c_1(M_i) \cdot \beta} (c_1(M_i) + kz)  \right) \cdot I^{X}_{\beta}. \]

\vspace{5pt}
\noindent 
\textbf{Toric varieties}(\cite{Givental}). Let $D_i, i=1, \ldots, n$ be the torus invariant divisors on the toric variety $V//T$.
\[ I_{\beta}^{V//T} = \prod_{i=1}^{n} \frac{ \prod_{k=-\infty}^{0} (D_i + kz)  }{ \prod_{k=-\infty}^{D_i \cdot \beta} (D_i + kz) }. \]

\begin{example}(Projective space $\p^{n-1}$) We have
$I^{\p^{n-1}}_{d} = ( \prod_{k=1}^{d} (H + kz)^{n} )^{-1}$.
\end{example}
\begin{example}(Grassmannian)
Let $M_{k \times n}$ be the space of $k \times n$-matrices acted on by $\GL(k)$ on the left.
Taking the determinant character,
the stable locus is the locus of matrices of full rank
and the associated GIT quotient  is the Grassmannian
\[ \Gr(k,n) = M_{k \times n} //_{\det} \GL(k). \]
The stable locus for the maximal torus $T \subset \GL(k)$ of diagonal matrices
is given by matrices where each row is non-zero. The abelian quotient is
\[ M_{k \times n} // T = \underbrace{\p^{n-1} \times \ldots \times \p^{n-1}}_{k \text{ times }} \]
The roots of $\GL(k)$ are $e_i^{\ast} - e_j^{\ast}$ and correspond to $\CO(H_i - H_j)$ where $H_i$ is the hyperplane class pulled back from the $i$-th factor.

The universal subbundle $\CU \to \BC^{n} \otimes \CO_{\Gr}$ on the Grassmannian corresponds to the inclusion of $G$-representations
\[ M_{k \times n} \times \BC^{k} \to M_{k \times n} \times \BC^{n} \]
where a column vector $w \in \BC^{k}$ is acting on by $g \cdot w := (g^t)^{-1} w$,
and $\BC^{n}$ carries the trivial representation.
The Pl\"ucker polarization on $\Gr(k,n)$ thus corresponds to
the line bundle $\CO(H_1 + \ldots + H_{k})$ on $(\p^{n-1})^k$.
Hence if we consider degree $d$ curves on the Grassmannian, in the abelian/non-abelian correspondence we have to sum over $(d_1, \ldots, d_{k})$ adding up to $d$.

Calculating the $I$-function is then easy. For example, for $k=2$ (and dropping the pullbacks $\xi^{\ast}$, $j^{\ast}$ from notation),
one obtains
\[ I^{\Gr(2,n)}_{d}
= \sum_{d=d_1 + d_2} (-1)^d \frac{ H_1 - H_2 + (d_1 - d_2)z }{H_1 - H_2} \frac{1}{ \prod_{k=1}^{d_1} (H_1 + kz)^{n} \prod_{k=1}^{d_2} (H_2 + kz)^n } \]
where the division by $H_1 - H_2$ is to take place formally.
\end{example}

\begin{example}(Fano variety of a cubic fourfold)
The Fano variety of a cubic fourfold $X \subset \Gr(2,6)$ is a zero locus of a section of $\mathrm{Sym}^{3}(U^{\vee})$.
On the abelian quotient $\p^5 \times \p^5$ this vector bundle corresponds to
\[ \CO(3 H_1) \oplus \CO(2 H_1 + H_2) \oplus \CO(H_1 + 2 H_2) \oplus \CO(3 H_2). \]
We find the $I$-function
\[
I^{X \subset \Gr(2,6)} = 
I^{\Gr(2,6)}_{d} \cdot \prod_{\substack{3 = i_1 + i_2 \\ i_1, i_2 \geq 0}} \prod_{k=1}^{i_1 d_1 + i_2 d_2} (i_1 H_1 + i_2 H_2 + kz).
\]
\end{example}

\begin{example}(A pencil of cubic fourfolds)
We consider a generic pencil of cubic fourfolds $\CX \subset \Gr(2,6) \times \p^1$.
Since $\CX$ is the zero locus of a generic section of the globally generated bundle $\mathrm{Sym}^3(U^{\vee}) \otimes \CO_{\p^1}(1)$,
it is smooth by a Bertini type argument.
The abelian quotient is $\p^5 \times \p^5 \times \p^1$.
Let $h$ be the hyperplane class on $\p^1$.
Then the $I$-function for the fiber part reads:
\begin{multline*}
I^{\CX}_{(d,0)}
= (-1)^d \sum_{d=d_1 + d_2} \frac{ H_1 - H_2 + (d_1 - d_2)z }{H_1 - H_2} \cdot \\
\frac{1}{ \prod_{k=1}^{d_1} (H_1 + kz)^{6} \prod_{k=1}^{d_2} (H_2 + kz)^6 }
\prod_{\substack{3 = i_1 + i_2 \\ i_1, i_2 \geq 0}} \prod_{k=1}^{i_1 d_1 + i_2 d_2} (i_1 H_1 + i_2 H_2 + h + kz)
\end{multline*}
\end{example}

\begin{example}(A pencil of Debarre-Voisin fourfolds)
We  consider a pencil $\CX \subset \Gr(6,10) \times \p^1$ of DV fourfolds
which is cut out by $\wedge^3 \CU^{\vee} \otimes \CO(1)$.
The abelian quotient is $(\p^9)^6 \times \p^1$.
We let $h$ be the hyperplane class of $\p^1$, and $H_i$ be the hyperplane class pulled back from the $i$-th copy of $\p^9$.
Then the $I$-function in the fiber class is:
\begin{align*}
I^{\CX}_{(d,0)}
= (-1)^d & \sum_{d=d_1 + \ldots + d_6} \left( \prod_{1 \leq i < j \leq 6} \frac{ H_i - H_j + (d_i - d_j)z }{H_i - H_j} \right) \\
& \times \prod_{i=1}^{6} \prod_{k=1}^{d_i} \frac{1}{ (H_i + kz)^{10} } \\
& \times \prod_{1 \leq i_1 < i_2 < i_3 \leq 6} \prod_{k=1}^{d_{i_1} + d_{i_2} + d_{i_3}} (H_{i_1} + H_{i_2} + H_{i_3} + h + kz)
\end{align*}
\end{example}

\subsection{$I$ and $J$ functions}
Given a GIT quotient $X$ as before,
let $t \in H^2(X,\BC)$ be any element (or a formal variable).
The big $J$-function (at $\epsilon = \infty$) with insertion $t$ is
\[
J^{\infty}(q,t,z) = e^{\frac{t}{z}} \left[ 1 + \sum_{\beta \neq 0} e^{\beta \cdot t} q^{\beta}
 \ev_{\ast}\left( \frac{ [ \Mbar_{0,1}( X, \beta) ]^{\text{vir}} }{z (z-\psi)} \right) \right]
\]
Expand the small $I$-function according to degree:
\[ I = I_0(q) + \frac{I_1(q)}{z} + \frac{I_2(q)}{z^2} + \ldots \, . \]
Then the mirror theorem \cite{CFK} states that:
\[ J^{\infty}\left( q, \frac{I_1(q)}{I_0(q)} ,z \right) = \frac{I(q,z)}{I_0(q)}. \]

By inverting this relation we can compute the descendent $1$-point invariants of $X$.
We sketch the details for the case of the Lefschetz pencil $\CX \subset \Gr(2,6) \times \p^1$ of Fano's. 
In this case (with $H$ the Pl\"ucker polarization on $\Gr(2,6)$) let us write
\[ I_0(q) = f_0(q), \quad I_1(q) = f_1(q) H + f_2(q) h. \]
Then with $t = I_1(q)/I_0(q)$ and since $\beta$ is fiber we get:
\[ e^{\beta \cdot t} q^{d} = \left( q e^{ \frac{f_1(q)}{f_0(q)} } \right)^d = Q^d \]
where we have identified $q^{\beta} = q^d$ and used the variable
\[ Q = q \exp\left( \frac{f_1(q)}{f_0(q)} \right). \]
Then we obtain the relation
\[
\exp\left( -\frac{I_1(q)}{I_0(q)} \frac{1}{z} \right) \frac{I^{\text{fib}}(q,z)}{I_0(q)} = 1 + \sum_{\beta \neq 0} Q^{\beta} \ev_{\ast}\left( \frac{ [ \Mbar_{0,1}( X, \beta) ]^{\text{vir}} }{z (z-\psi)} \right)
\]
where $I^{\text{fib}}(q,z)$ stands for the $I$-functions involving only fiber classes $\beta$.

\section{Results}
\subsection{Cubic fourfolds}
We consider a generic pencil of Fano varieties of cubic fourfolds
\[ \CX \subset \Gr(2,6) \times \p^1. \]
This defines a $1$-parameter family $\pi : \CX \to \p^1$ polarized by the Pl\"ucker embeddings.
The family has precisely $192$ singular fibers $\CX_t$,
which are irreducible varieties with ordinary  double point singularities along a smooth K3 surface (and smooth elsewhere) \cite{CG}.
The blowup $\mathrm{Bl}_S \CX_t$ along the singular locus is isomorphic to the Hilbert scheme $\Hilb_2(S)$,
and the blowdown map contracts a $\p^1$-bundle over $S$ along its fibers which are $(-2)$-curves \cite[Sec.6.3]{Hassett}.
The map $\CX \to \p^1$ is a ordinary double point degeneration to $\CX_t$.

To obtain a family of smooth holomorphic-symplectic manifolds, we follow the arguments of Maulik and Pandharipande \cite[Sec.5.1]{MP}.
We choose a double cover
\[ \epsilon : C \to \p^1 \]
which is ramified along the $192$ base points of nodal fibers. The family
\[ \epsilon^{\ast} \CX \to C \]
then has double point singularities along the surfaces $S$ which can be resolved by a small resolution
\[ \tilde{\pi} : \tilde{\CX} \to C. \]
The family $\tilde{\pi}$ is a $1$-parameter family of quasi-polarized $K3^{[2]}$-type varieties,
polarized by the pullback of the Pl\"ucker polarization.

The Noether-Lefschetz numbers of the family $\pi$ in terms of $\tilde{\pi}$ are then:
\[ \NL_{s,d}^{\pi} = \frac{1}{2} \NL_{s,d}^{\tilde{\pi}}. \]

We also have the following comparison of Gromov-Witten invariants of the total spaces of $\CX$ and $\tilde{\CX}$ in fiber classes.
We consider $1$-pointed invariants to simplify the notation.

\begin{lemma} For any $i, \alpha$,
\[ \langle \alpha; H^i \rangle^{\CX, \textup{mc}}_{g,d} = \frac{1}{2} \langle \alpha; H^i \rangle^{\widetilde{\CX}, \textup{mc}}_{g,d}. \]
\end{lemma}
\begin{proof}
We need to prove that
\[ \langle \alpha; H^i \rangle^{\CX}_{g,d} = \frac{1}{2} \langle \alpha; H^i \rangle^{\widetilde{\CX}}_{g,d}. \]
This follows from the same argument as in \cite[Lem. 4]{MP}:
The conifold transition is taken relative to the K3 surface $S$.
The extra components which appear in the degeneration argument is a bundle (with fiber $\p(\CO_{\p^1}(-1) \oplus \CO_{\p^1}(-1) \oplus \CO_{\p^1})$ or a quadric in $\p^4$)
over the K3 surface $S$. Because of the existence of the symplectic form, it follows that the curve classes which may contribute non-trivially have to be fiber classes.
The argument of \cite{LR} then goes through without change.
\end{proof}

By the lemma the Gromov-Witten/Noether-Lefschetz relation of Propsition~\ref{prop:GWNLrelation} extends to the family $\pi$.
Specializing to genus $0$ we obtain that
\begin{equation}
\blangle \alpha ; H^i \brangle^{\CX, \text{mc}}_{0,d}
=
\sum_{m,s} \blangle \alpha ; H^i \brangle^{X,\text{mc}}_{0,m,s,d} \NL^{\pi}_{m,s,d}. \label{abcc}
\end{equation}
The left hand side can be computed using the mirror symmetry formulas of Section~\ref{sec:mirror_symmetry}.
The primitive invariants appearing on the right hand side are given by Remark~\ref{rmk:primitive evaluations}.
The Noether-Lefschetz numbers $\NL_{s,d}$ and hence their refinements $\NL_{m,s,d}$
are determined by \cite{LZ}and the formulas in Section~\ref{subsec:refined NL}.
By using a computer (see the author's website for the code)
one finds that for degree $6, 8, 9, 15$ this equation uniquely determines the invariants $f_{\beta}, g_{\beta}$ for $\beta = m \alpha$
in cases $(m,\alpha) \in \{ (2,0), (2,3/2), (3,3/2), (5,3/2) \}$.
Moreover, one checks then that for these degrees we have
\begin{equation} \label{abccheck} 
\blangle \alpha ; H^i \brangle^{\CX, \text{mc}}_{0,d}
=
\sum_{s} \NL_{s,d} \cdot \blangle \alpha ; H^i \brangle^{X,\text{mc}}_{0,1,s,d}.
\end{equation}
which implies that Conjecture~\ref{conjecture:MC_intro} holds in these cases.
Together with Proposition~\ref{prop:first cases} this proves Proposition~\ref{prop:first cases intro}.
(As mentioned in the introduction, we have checked \eqref{abccheck} up to degree $38$, which provides plenty of evidence for Conjecture~\ref{conjecture:MC_intro}.)

\subsection{Debarre-Voisin fourfolds}
We consider a generic pencil of Debarre-Voisin fourfolds
\[ \CX \subset \Gr(6,10) \times \p^1, \quad \pi : \CX \to \p^1. \]
The case is very similar to the case of cubic fourfolds.
As shown in Appendix~\ref{appendix:Song} (by J. Song)
we have the same description of the singular fibers as in the Fano case.
In particular, we may use the same double cover construction
and conclude the Gromov-Witten/Noether-Lefschetz relation \eqref{abcc} for $\pi$.

We want to determine the generating series of Noether-Lefschetz numbers
\[ \varphi(q) = \sum_{D \geq 0} q^{D/11} \NL^{\pi}(D) \]
where $D$ runs over squares modulo $11$, and we used the notation of Section~\ref{ex: prime discriminant}.
Recall that we have
\[ \NL^{\pi}_{s,d} = \NL^{\pi}\left(D \right), \  \text{ where }\  D = -\frac{1}{4} \det \begin{pmatrix} 22 & d \\ d & s \end{pmatrix}. \]

We first prove the following basic invariants:
\begin{lemma} \label{lemma:ic}
$\NL^{\pi}(0) = -10$ and $\NL^{\pi}(11) = 640$. 
\end{lemma}
\begin{proof}
We have
\[ \NL^{\pi}(0) = \int_{\p^1} \iota_{\pi}^{\ast} c_1(\CK^{\ast}) \]
where $\CK \to \CD_{L} / \Gamma_M$ is the descent of the tautological bundle $\CO(-1)$.
It is well-known that $\iota_{\pi}^{\ast} \CK$ corresponds to the Hodge bundle $\pi_{\ast} \Omega^2_{\pi}$.
Hence $\iota_{\pi}^{\ast} \CK^{\ast}$ is isomorphic to
\[ \CL = R^2 \pi_{\ast} \CO_{\CX} \]
which has fiber $H^2(\CX_t,\CO_{\CX_t})$ over $t \in \p^1$. In $K$-theory we have
\[ R \pi_{\ast} \CO = \CO_{\p^1} + \CL + \CL^{\otimes 2} \]
By a Riemann-Roch calculation (using the software package \cite{Chow}) we find that
\[ 3 c_1(\CL) = \int \ch_1( R \pi_{\ast} \CO ) = \int \pi_{\ast}( \td_{\CX} / \td_{\p^1} ) = -30 \]

The number $\NL^{\pi}(11) = \NL_{-2,0}$ is the number of singular fibers.
To compute these, we recall that the singular locus of every singular fiber is a smooth K3 surface
and the blowup along the singular locus has exceptional divisor a $\p^1$-bundle over the K3 surface.
Hence the topological Euler characteristic of a singular fiber is $300$. 
By a standard computation (using \cite{Chow}) the topological Euler number of the total family is $e(\CX) = -14712$.
Hence if $\delta$ is the number of singular fibers we get
\[ -14712 = e(\CX) = 324 (2-\delta) + \delta \cdot 300, \quad \text{ hence } \delta = 640. \]
The last part also follows from \cite[Proof of Prop.3.1]{DV}.
\end{proof}

To further constrain the Noether-Lefschetz numbers we argue as follows.
By a computer check (see again the author's webpage)
the Gromov-Witten/Noether-Lefschetz relation
\begin{equation}
\blangle H^3 \brangle^{\CX, \text{mc}}_{0,d}
=
\sum_{m, s} \blangle H^3 \brangle^{X,\text{mc}}_{0,m,s,d} \NL^{\pi}_{m,s,d}. \label{abc}
\end{equation}
involves for $d \leq 13$ only terms for which the multiple cover conjecture is known by Proposition~\ref{prop:first cases}.
Hence for $d \leq 13$ we may rewrite it 
\[
\blangle H^3 \brangle^{\CX, \text{mc}}_{0,d}
=
\sum_{s} \NL_{s,d} \cdot \blangle H^3 \brangle^{X,\text{mc}}_{0,1,s,d}.
\]
The left hand side can be computed using the mirror symmetry formalism. The primitive invariants on the right are given in Remark~\ref{rmk:primitive evaluations}.
For $1 \leq d \leq 5$ 
one obtains:
\begin{equation} \label{eqnsa2}
\begin{aligned}
0 & = 264 \NL(3) \\
130680 & = 3960 \NL(1) + 132 \NL(12) \\
0 & = 792 \NL(5) \\
3020160 & = 264 \NL(15) + 7920 \NL(4) \\
0 & = 1320 \NL(9)
\end{aligned}
\end{equation}
Using equations \eqref{eqnsa2} and $\NL(0) = -10$, and employing Williams' program \cite{Williams} we find that:
\begin{equation} \NL(1) = \NL(3) = \NL(4) = \NL(5) = \NL(9) = 0. \label{vanishing} \end{equation}
This in turn determines the modular form $\Phi^{\pi}$ uniquely.
(One independently checks that indeed $\NL(11) = 640$ matches the second result of Lemma~\ref{lemma:ic}.)

Theorem~\ref{thm:DV NL series} follows from this, from Proposition~\ref{prop:phi modularity} and straightforward linear algebra
(by \cite[Sec.12]{Borcherds} we have that $E_1, \Delta_{11}, E_3$ generate the ring of modular forms for character $\chi_{11}$:
\[ \bigoplus_{k \geq 0} \mathrm{Mod}_k(\Gamma_0(11), \chi_{11}^{k}) = \BC[ E_1, \Delta_{11}, E_3 ] / (\text{relations})\ ). \]

\begin{proof}[Proof of Corollary~\ref{corHLS}]
Define the Noether-Lefschetz numbers of first type: 
\[ \CC^{\pi}_{2e} = \frac{1}{2} \int_{C} \iota_{\tilde{\pi}}^{\ast} \CC_{2e}. \]
These are related to the $\NL^{\pi}(D)$ by Proposition~\ref{prop:NL first and second}.
The proof hence follows from \eqref{vanishing} and Lemma~\ref{lemma:123} below.
\end{proof}

\begin{lemma} \label{lemma:123}
$\CC_{2e}$ is HLS if and only if $\CC^{\pi}_{2e} = 0$.
\end{lemma}
\begin{proof}
Let $\CM_{DV}' \subset \CM_{DV}$ be the open locus consisting of Debarre-Voisin varieties
which are either smooth or singular with ordinary double point singularity along a smooth K3 surface (which holds on an open subset of the irreducible discriminant divisor \cite{BS}).
By \cite{DV}, the complement of $\CM_{DV}'$ has codimension $2$.
The period map extends to a morphism $p : \CM_{DV}' \to \CM_H \subset \overline{\CM}_H$.
Let also $\tilde{p} : \widetilde{\CM}_{DV} \to \overline{\CM}_H$ be the resolution of the rational period map $\CM_{DV} \dashrightarrow \overline{\CM}_H$.
We view $\CM_{DV}'$ as a open subvariety of $\widetilde{\CM}_{DV}$,
\[ \CM_{DV}' \overset{j}{\hookrightarrow} \widetilde{\CM}_{DV} \xrightarrow{\tilde{p}} \overline{\CM}_H. \]

We need the following basic fact: Assume $D = \tilde{p}_{\ast} E$ for some irreducible divisor $E \subset \widetilde{\CM}_{DV}$.
Since $\tilde{p}$ is birational \cite{OGrady}, by Zariski's main theorem \cite[Cor.11.4]{Hartshorne} $E$ is the unique irreducible divisor in $\widetilde{\CM}_{DV}$ which maps to $D$
(otherwise, the generic point of the image would have more than two preimages, hence the fiber would not be connected) and by the same argument the map $E \to D$ is birational. Hence $\tilde{p}^{\ast} \tilde{p}_{\ast} E = E$. This yields
\[ p^{\ast} D = j^{\ast} \tilde{p}^{\ast} D =  j^{\ast} \tilde{p}^{\ast} \tilde{p}_{\ast} E = j^{\ast} E \]
We find that $D$ is HLS (term on the right vanishes for some necessarily unique $E$) if and only if $p^{\ast} D = 0$ in $A^1(\CM_{DV}')$.

There is a $\mathrm{SL}(V_{10})$-bundle $\pi : U \to \CM_{DV}'$ for an open $U \subset \p(\wedge^3 V_{10}^{\ast})$ with complement of codimension $\geq 2$.
Hence we also have $p^{\ast} D = 0$ if and only if $\pi^{\ast} p^{\ast} D = 0$ if and only if $\pi^{\ast} p^{\ast} D \cdot L = 0$ for a generic line $L$ in $U$.
\end{proof}

%

\appendix
\section{A multiple cover rule for abelian surfaces}
In this appendix we state a conjectural rule that
expressed reduced Gromov-Witten invariants of an abelian surfaces for any curve class $\beta$
in terms of invariants for which $\beta$ is primitive.
The conjectural formula extends a proposal of \cite{BOPY} for the abelian surface analogue of the Katz-Klemm-Vafa formula.
As in the hyperk\"ahler case the conjecture can be reinterpreted as saying that after
subtracting multiple covers, the Gromov-Witten invariants are independent of the divisibility.

\subsection{Monodromy}
Recall that the cohomology of an abelian surface is described by
\[ H^i(A,\BQ) = \bigwedge^i H^1(A,\BQ). \]
The class of a point $\pt \in H^4(A,\BZ)$ thus defines a canonical element
\[ \pt \in \bigwedge^4 H^1(A,\BQ). \]
An isomorphism of abelian groups $\varphi : H^1(A,\BZ) \to H^1(A',\BZ)$
extends naturally to a morphism of the full cohomology $H^{\ast}(A,\BZ)$
by setting $\varphi|_{H^i(A,\BZ)} = \wedge^i \varphi$.
One has that $\varphi$ is a parallel transport operator (i.e.\, the parallel transport along a deformation from $A$ to $A'$ through complex tori)
if and only if $\varphi$ preserves the canonical element \cite[Sec.1.10]{BirkLang}.

The Zariski closure of the space of parallel transport operators 
is the set of $\BC$-vector spaces homomorphisms:
\[ M_{A,A'} = \{ \varphi : H^1(A,\BC) \to H^1(A',\BC) | \varphi(\pt) = \pt' \}. \]
It follows that the induced map $\wedge^2 \varphi: H^2(A,\BC) \to H^2(A',\BC)$
preserves the canonical inner product.
If $A'=A$ the above just says that the monodromy group is $\mathrm{SL}(H^1(A,\BZ))$
and its Zariski closure $\mathrm{SL}(H^1(A,\BC))$.

\subsection{Multiple cover rule}
Let $\beta \in H_2(A,\BZ)$ be an effective curve class.
For any divisor $k | \beta$ choose an abelian variety $A_k$ and a morphism $\varphi_k : H^1(A,\BR) \to H^1(A_k,\BR)$
preserving the canonical element such that the induced morphism
\[ \varphi_k = \bigoplus_{i} \wedge^i \varphi_k : H^{\ast}(A,\BR) \to H^{\ast}(A_k, \BR) \]
takes $\beta/k$ to a \emph{primitive} effective curve class.

Let also $\alpha \in H^{\ast}(\Mbar_{g,n})$ be a tautological class
and $\gamma_i \in H^{\ast}(A,\BR)$ be arbitrary insertions.

\begin{conjecture} \label{conjecture:abelian surfaces}
For any effective curve class $\beta \in H_2(A,\BZ)$.
\begin{multline*}
\Big\langle \alpha ; \gamma_1, \ldots, \gamma_n \Big\rangle^A_{g,\beta}
=
\sum_{k | \beta} k^{3g-3+n - \deg(\alpha)} \Big\langle \alpha ; \varphi_k(\gamma_1), \ldots, \varphi_k(\gamma_n) \Big\rangle^{A_k}_{g,\varphi_k(\beta/k)}.
\end{multline*}
\end{conjecture}

\subsection{Example}
We apply the conjectural multiple cover formula to the analogue of the Katz-Klemm-Vafa formula
for abelian surfaces which is the integral
\[ N^{\text{FLS}}_{g,\beta} = \int_{ [ \Mbar_{g,n}(A,\beta)^{\text{FLS}} ]^{\text{red}}} (-1)^{g-2} \lambda_{g-2} \]
where $\Mbar_{g,n}(A,\beta)^{\text{FLS}}$ is the substack of $\Mbar_{g,n}(A,\beta)$ that maps with image in a fixed linear system (FLS), see \cite{BOPY}.

To apply the multiple cover rule we specialize to $A = E \times E'$.
Consider symplectic bases
\[ \alpha_1, \beta_1 \in H^1(E,\BZ), \quad \alpha_2, \beta_2 \in H^1(E',\BZ) \]
which give a basis of $H^1(A,\BZ)$ (we omit the pullback), and let
\[ \omega_1 = \alpha_1 \beta_1,\quad \omega_2 = \alpha_2 \beta_2  \in H^2(A,\BZ). \]
We take $\beta = (d_1,d_2) := d_1 \omega_1 + d_2 \omega_2$.
For every $k | \mathrm{gcd}(d_1,d_2)$ define $\varphi_k \in \mathrm{SL}(H^1(A,\BQ))$ by
\[ \alpha_1 \mapsto \alpha_1,\quad \beta_1 \mapsto \frac{k}{d_1} \beta_1,\quad \alpha_2 \mapsto \alpha_2,\quad \beta_2 \mapsto \frac{d_1}{k} \beta_2. \]
The extension to the full cohomology satisfies
\[ \varphi_k( \beta/k ) = \omega_1 + \frac{d_1 d_2}{k^2} \omega_2 = (1, d_1 d_2/k^2). \]

Recall the result of Bryan (\cite[Sec.3.2]{BOPY}) that:
\[ N^{\text{FLS}}_{g,\beta} = \Big\langle (-1)^{g-2} \lambda_{g-2} ; \prod_{i=1}^{4} \xi_i \Big\rangle^A_{g,\beta} \]
where we can take
\[ (\xi_1,\xi_2,\xi_3,\xi_3)  = (\omega_1 \alpha_2, \omega_1 \beta_2, \alpha_1 \omega_2, \beta_1 \omega_2). \]
Conjecture~\ref{conjecture:abelian surfaces} then implies:
\begin{align*}
 N^{\text{FLS}}_{g,(d,d')}
& = \sum_{k|\beta} k^{2g-3+6} \blangle (-1)^{g-2} \lambda_{g-2} ; \prod_{i=1}^{4} \varphi_k(\xi_i) \brangle_{g, (1, dd'/k^2)} \\
& = \sum_{k|\beta} k^{2g+3}  N^{\text{FLS}}_{g,(1,dd'/k^2)}
\end{align*} 
which matches precisely Conjecture A in \cite{BOPY}.

\section{Comparision with Gopakumar-Vafa invariants} \label{sec:comparision with GV}
For a K3 surface $S$ and effective curve class $\beta \in H_2(S,\BZ)$ consider the Gromov-Witten invariant
\[ R_{g,\beta} = \int_{ [ \Mbar_{g,n}(S,\beta) ]^{\text{red}} } (-1)^g \lambda_g. \]
Fix any primitive effective class $\alpha \in H_2(S,\BZ)$. Define the generating series
\[
F_{\alpha} = \sum_{g \geq 0} \sum_{m>0} R_{g, m \alpha} u^{2g-2} v^{m},
\]
where $u,v$ are formal variables. Following \cite{PT_KKV} the Gopakumar-Vafa invariants $r_{g,m \alpha} \in \BQ$
of the K3 surface $S$ 
are defined by the equality:
\begin{equation} \label{defn BPS}
F_{\alpha} = \sum_{g \geq 0} \sum_{m > 0} r_{g,m \alpha} \sum_{k>0} \frac{1}{k} \left( \frac{ \sin(ku/2) }{2} \right)^{2g-2} v^{km}.
\end{equation}

Recall from the introduction (Section~\ref{sec:intro K3}) the numbers:
\[ \widetilde{r}_{g,\beta} = \sum_{k | \beta} k^{2g-3} \mu(k) R_{g, \beta/k} \]

We have the following connection between the invariants $r_{g,\beta}$ and $\widetilde{r}_{g, \beta}$:
For any $g$ consider the expansion
\begin{equation} \label{agg def}
\left( \frac{1}{2} \sin(u/2) \right)^{2g-2} = \sum_{\tilde{g}} a_{g, \tilde{g}} u^{2 \tilde{g}-2}.
\end{equation}
\begin{lemma} \label{lemma:r rtilde} For any $g \geq 0$ we have the upper-triangular relation:
\begin{equation} \label{r rtilde rel}
\tilde{r}_{\tilde{g},\beta} = \sum_{\tilde{g}} a_{g, \tilde{g}} r_{g, \beta}.
\end{equation}
\end{lemma}

In \cite{PT_KKV} it was shown that the $r_{g,\beta}$ do not depend on the divisibility of the curve class $\beta$.
By \eqref{r rtilde rel} we find that also $\tilde{r}_{g,\beta}$ does not depend on the divisibility,
as claimed in Theorem~\ref{thm:PT}.
Since \eqref{r rtilde rel} is upper-triangular we see that also the converse holds,
i.e. $r_{g,\beta}$ does not depend on the divisibility if and only if the same holds for $\tilde{r}_{g,\beta}$.

\begin{proof}[Proof of Lemma~\ref{lemma:r rtilde}]
Let us define 
\[ \hat{r}_{\tilde{g},\beta} := \sum_{\tilde{g}} a_{g, \tilde{g}} r_{g, \beta}. \]
Inserting \eqref{agg def} into \eqref{defn BPS} we get:
\begin{align*}
F_{\alpha}
&= \sum_{g \geq 0} \sum_{m > 0} r_{g,m \alpha} \sum_{k>0} \frac{1}{k} \sum_{\tilde{g}} a_{g,\tilde{g}} k^{2 \tilde{g}-2} u^{2 \tilde{g}-2} v^{km} \\
&= \sum_{\tilde{g} \geq 0} \sum_{m > 0} \sum_{k>0} k^{2 \tilde{g}-3} \hat{r}_{g, m \alpha} u^{2 \tilde{g}-2} v^{km}.
\end{align*}
Taking the $v^n u^{2g-2}$ coefficient this shows that
\[
R_{g, n\alpha} = \sum_{k|n} k^{2 g-3} \hat{r}_{g,n/k \alpha}.
\]
By M\"obius inversion (i.e. using the identity $\sum_{d|n, d>0} \mu(d) = \delta_{n1}$), 
we get 
$\hat{r}_{g,n \alpha} = \sum_{k | n} k^{2g-3} \mu(k) R_{g, n \alpha/k}$, so $\widetilde{r}_{g, n\alpha} = \hat{r}_{g, n \alpha}$.
\end{proof}

By \cite{PT_KKV} all the $r_{g,\beta}$ are integers.
It would be interesting to find integer-invariants which
underlie the Gromov-Witten invariants of hyperk\"ahler varieties in dimension $>2$.
For hyperk\"ahler fourfolds, a partial proposal is discussed in \cite{COT}.

\section{Geometry of a general singular Debarre--Voisin fourfold\\
{\em by Jieao Song}}
\label{appendix:Song}

We give a description for the singularities of a general singular
Debarre--Voisin variety. In the notation of~\cite{BS}, the class of the
trivector $\sigma$ defining a general such Debarre--Voisin variety $X_6^\sigma$
lies in the divisor $\cD^{3,3,10}$: there exists a unique $3$-dimensional
subspace $V_3\subset V_{10}$ such that $\sigma$ satisfies the degeneracy
condition $\sigma(V_3,V_3,V_{10})=0$. Under the period map, this divisor
corresponds to the Heegner divisor\footnote{The divisor $\cD_{22}$ is denoted by $\CC_{22}$ in the main body of the text.} $\cD_{22}$ in the period domain. We obtained
the following description for the set-theoretical singular locus of
$X_6^\sigma$ in~\cite[Proposition~2.4]{BS}.

\begin{prop}
Let $[\sigma]\in \cD^{3,3,10}$ be general, so that there exists a unique
$3$-dimensional subspace $V_3\subset V_{10}$ with $\sigma(V_3,V_3,V_{10})=0$.
Set-theoretically, the singular locus of $X_6^\sigma$ is
\[
S\coloneqq\setmid{[V_6]\in X_6^\sigma}{V_6\supset V_3},
\]
which is a K3 surface of degree $22$.
\end{prop}

We prove the following stronger result, following the idea in
\cite[Lemma~6.3.1]{Hassett}, where a similar result is proved for the variety
of lines of a nodal cubic hypersurface. We shall see that the two cases share
some surprising similarities.

\begin{prop}
Let $\sigma$ be as in the previous proposition. For the associated
Debarre--Voisin variety $X_6^\sigma$, the singularities along the degree-$22$
K3 surface~$S$ are codimension-$2$ ordinary double points.
More precisely, by blowing up the singular locus~$S$, we get a smooth
hyperk\"ahler fourfold of $\KKK^{[2]}$-type, and the exceptional divisor is a
conic fibration over~$S$.
\end{prop}

\begin{proof}
We briefly recall the argument for the nodal cubic: for a cubic
$X\subset\bP^5=\bP(V_6)$ containing a node $p\coloneqq[V_1]$, the projectivized
normal cone $\bP C_p X$ is a quadric hypersurface $Q$ in $\bP T_p
\bP^5=\bP(V_6/V_1)$, and the varieties of lines $F\subset \Gr(2,V_6)$ is
singular along a K3 surface $S$ parametrizing lines in $X$ passing through $p$.
Instead of blowing up $S$ in $F$, Hassett considered studying the ambient
Grassmannian $\Gr(2,V_6)$ and blowing up the Schubert variety
$\Sigma\coloneqq\bP(V_6/V_1)\subset \Gr(2,V_6)$, which parametrizes all lines
in $\bP(V_6)$ passing through $p$. This gives the following Cartesian diagram
\[
\begin{tikzcd}
\tilde F\coloneqq\Bl_S F \ar[r, hook]\ar[d]& \Bl_{\Sigma}\Gr(2,V_{6})\ar[d]\\
F \ar[r,hook] & \Gr(2,V_{6}).
\end{tikzcd}
\]
For a given point $x\coloneqq[V_2]\in S$, we get one distinguished point
$y\coloneqq[V_2/V_1]$ in $\bP^4=\bP(V_6/V_1)$ that lies on the quadric $Q$. The
projectivized normal space $\bP\cN_{\Sigma/\Gr(2,V_6),x}$ can be identified
with $\bP(V_6/V_2)$, which parametrizes lines in $\bP^4=\bP(V_6/V_1)$ passing
through the point $y$, and the projectivized normal cone $\bP C_{S,x}F$ is
given by the subscheme parametrizing such lines that are also entirely
contained in the quadric threefold $Q$, in other words, lines in $Q$ passing
through a given point. This condition gives a smooth conic curve, so the
singularities of $F$ along $S$ are indeed codimension-2 ordinary double points.


We use a similar argument to study the singular Debarre--Voisin variety
$X_6^\sigma$. By assumption, the hyperplane section $X_3^\sigma$ admits an
ordinary double point at $[V_3]$, so its tangent cone at $[V_3]$ is a smooth
quadric hypersurface~$Q$ in the projectivization of the tangent space
\[\bP T_{[V_3]}\Gr(3,V_{10})\simeq\bP\Hom(V_3,V_{10}/V_3)=:
\bP(T_{21})=\bP^{20}.\]
For a given $x\coloneqq[V_6]\in S$, the projective space
$\bP\Hom(V_3,V_6/V_3)=:\bP(T_9)=\bP^8$ gives a distinguished linear subspace
contained in~$Q$.

Following the proof of Hassett, instead of blowing up $S$ in $X_6^\sigma$, we
consider the ambient Grassmannian $\Gr(6,V_{10})$ and blow up the entire
Schubert variety
\[
\Sigma\coloneqq\setmid{[V_6]\in\Gr(6,V_{10})}{V_6\supset V_3}
\simeq\Gr(3,V_{10}/V_3),
\]
which is smooth of codimension~12.
We have the following description for its normal bundle in $\Gr(6,V_{10})$:
\[
\cN_{\Sigma/\Gr(6,V_{10})}=\Hom(\cU_6,\cQ_{10/6})/\Hom(\cU_6/V_3,
\cQ_{10/6})\simeq \Hom(V_3,\cQ_{10/6}),
\]
where we denote by $\cU_6$ and $\cQ_{10/6}$ the restrictions to $\Sigma$ of the
two tautological bundles on $\Gr(6,V_{10})$.
For the given point $x\in S$, the projectivization of the normal space is
therefore an 11-dimensional projective space
\[
\bP\cN_{\Sigma/\Gr(6,V_{10}),x}\simeq \bP\Hom(V_3,V_{10}/V_6)\simeq
\bP(T_{21}/T_9),
\]
where we recall that $T_{21}$ is the tangent space of $\Gr(3,V_{10})$ at
$[V_3]$, and $T_9$ is the tangent space of $\Gr(3,V_6)$ at~$[V_3]$, viewed as a
subspace of $T_{21}$.

Consider the proper transform of $X_6^\sigma$ denoted by $\tilde
X_6^\sigma$. We have the following Cartesian diagram
\[
\begin{tikzcd}
\tilde X_6^\sigma \ar[r, hook]\ar[d]& \Bl_{\Sigma}\Gr(6,V_{10})\ar[d]\\
X_6^\sigma \ar[r,hook] & \Gr(6,V_{10}).
\end{tikzcd}
\]
Consequently we get a natural closed embedding of the projectivized normal cone
\[
\bP C_{S,x}X_6^\sigma\into \bP\cN_{\Sigma/\Gr(6,V_{10}), x}\simeq\bP(T_{21}/T_9).
\]
The total projective space $\bP(T_{21}/T_9)$ parametrizes 9-dimensional linear
subspaces of $\bP(T_{21})$ that contains the distinguished $\bP^8=\bP(T_9)$,
and the projectivized normal cone $\bP C_{S,x} X_6^\sigma$ can then be
identified with the subscheme that parametrizes such $\bP^9$ that are also
contained in the quadric $Q$. In other words, it parametrizes 9-dimensional
linear subspaces in a 19-dimensional quadric containing a fixed $\bP^8$. This
is again a smooth conic curve, just like in the nodal cubic case. Thus the
singularities of $X_6^\sigma$ along $S$ are indeed codimension-2 ordinary
double points, and $\tilde X_6^\sigma$ is smooth.


Finally, we show that the resolution $\tilde X_6^\sigma$ that we obtained has
trivial canonical class. Since $X_6^\sigma$ is birational to the Hilbert square
$S^{[2]}$, this will then force $\tilde X_6^\sigma$ to be a smooth hyperk\"ahler
fourfold of $\KKK^{[2]}$-type.

We denote by $E$ the exceptional divisor for the
blowup $\Bl_\Sigma\Gr(6,V_{10})\to \Gr(6,V_{10})$, and by $D$ the exceptional
divisor for the blowup $\tilde X_6^\sigma\to X_6^\sigma$.
The divisor $D$ can be identified with the projectivized normal cone $\bP C_S
X_6^\sigma$, so the morphism $D\to S$ is a conic fibration by the above
analysis. By construction, the Zariski open subset $\tilde X_6^\sigma\setminus
D$ is isomorphic to the smooth locus $X_6^\sigma\setminus S$. The latter has
trivial canonical class since it is the regular zero-locus of $\sigma$ viewed
as a section of the vector bundle $\bw3\cU_6^\vee$. Therefore, the canonical
divisor $K_{\tilde X_6^\sigma}$ is linearly equivalent to some multiple of $D$.
We write $K_{\tilde X_6^\sigma}=m D$, and it remains to show that $m=0$.

Since $D\to S$ is a smooth conic fibration in the projectivized normal bundle
$E\to \Sigma$, the relative $\cO(-1)$ of $E\to \Sigma$ restricts to the
relative canonical bundle of $D\to S$. Note that by the Leray--Hirsch theorem,
this bundle is necessarily non-trivial. Since $E$ is the exceptional divisor,
the relative $\cO(-1)$ on $E$ is given by $\cO_E(E)$, so we have
\[
\omega_{D/S}\simeq\cO_E(E)|_D.
\]
Using the fact that $S$ is a K3 surface and that $\cO_E(E)|_{\tilde
X_6^\sigma}\simeq \cO_{\tilde X_6^\sigma}(D)$, this gives
\[
\omega_D\simeq\omega_{D/S}\simeq\cO_{\tilde X_6^\sigma}(D)|_D,\quad\text{hence}
\quad K_D=D|_D,
\]
which in particular must be non-trivial.

On the other hand, by the adjunction formula we have
\[
K_D\simeq(K_{\tilde X_6^\sigma}+D)|_D=(m+1)D|_D.
\]
Thus we may conclude that $m=0$, and $K_{\tilde X_6^\sigma}$ is indeed trivial.
\end{proof}

\begin{rmk}
Contrary to the nodal cubic case, the resolution $\tilde X_6^\sigma$ obtained
is not isomorphic to the Hilbert scheme $S^{[2]}$, even for a generic member of
the family. This can be seen by studying the chamber decomposition for a
generic $S^{[2]}$ with Picard rank 2: one may find exactly two chambers in the
movable cone, corresponding to $S^{[2]}$ and a second birational model; the
Pl\"ucker polarization pulled back to $S^{[2]}$ via the birational map is equal
to $10H-33\delta$ and not nef (see for example \cite[Table~1]{DHOV}), so we may
conclude that $\tilde X_6^\sigma$ is the second birational model. The two
models are related by a Mukai flop, and it would be interesting to see this
geometrically.
\end{rmk}


\end{document}